\documentclass[12pt,reqno]{amsart}
\usepackage{amsthm,amsfonts,amssymb,euscript}

\newcommand{\bea}{\begin{eqnarray}}
\newcommand{\eea}{\end{eqnarray}}
\def\beaa{\begin{eqnarray*}}
\def\eeaa{\end{eqnarray*}}
\def\ba{\begin{array}}
\def\ea{\end{array}}
\def\be#1{\begin{equation} \label{#1}}
\def \eeq{\end{equation}}

\def\be{{\beta}}

\def\eps{\epsilon}

\def\al{\alpha}

\def\c{\cdot}

\def\H{{\mathbb{H}}}
\def\R{{\mathbb{R}}}
\def\C{{\mathbb{C}}}
\def\S{{\mathbb{S}}}

\def\D{{\bf D}}

\def\c{{\bf c}}
\def\g{{\bf g}}

\DeclareMathOperator{\ch}{ch}
\DeclareMathOperator{\sh}{sh}

\newtheorem{theorem}{Theorem}[section]
\newtheorem{lemma}[theorem]{Lemma}
\newtheorem{proposition}[theorem]{Proposition}
\newtheorem{corollary}[theorem]{Corollary}
\newtheorem{definition}[theorem]{Definition}
\newtheorem{remark}[theorem]{Remark}

\numberwithin{equation}{section}

\begin{document}

\title[Energy-critical defocusing NLS]{On the global well-posedness of energy-critical Schr\"{o}dinger equations in curved spaces}

\author{Alexandru D. Ionescu}
\address{University of Wisconsin--Madison}
\email{ionescu@math.wisc.edu}

\author{Benoit Pausader}
\address{Brown University}
\email{benoit.pausader@math.brown.edu}

\author{Gigliola Staffilani}
\address{Massachusetts Institute of Technology}
\email{gigliola@math.mit.edu}

\thanks{The first author was supported in part by a Packard Fellowship. The third author was supported in part by NSF Grant DMS 0602678. The second and third author thank the MIT/France program during which this work was initiated.}

\begin{abstract}
In this paper we present a method to study global regularity properties of solutions of large-data critical Schr\"{o}dinger equations on certain noncompact Riemannian manifolds. We rely on concentration compactness arguments and a global Morawetz inequality adapted to the geometry of the manifold (in other words we adapt the method of Kenig-Merle \cite{KeMe} to the variable coefficient case), and a good understanding of the corresponding Euclidean problem (in our case the main theorem of Colliander-Keel-Staffilani-Takaoka-Tao \cite{CKSTTcrit}). 

As an application we prove global well-posedness and scattering in $H^1$ for the energy-critical defocusing initial-value problem
\begin{equation*}
(i\partial_t+\Delta_\g)u=u|u|^{4},\qquad u(0)=\phi,
\end{equation*}
on the hyperbolic space $\H^3$.  
\end{abstract}
\maketitle
\tableofcontents

\section{Introduction}\label{Intro}

The goal of this paper is to present a somewhat general method to prove global well-posedness of critical\footnote{Here critical refers to the fact that when $(M,{\bf g})=(\mathbb{R}^3,\delta_{ij})$, the equation and the control (here the energy) are invariant under the rescaling $u(x,t)\to \lambda^\frac{1}{2}u(\lambda x,\lambda^2 t)$.}
nonlinear Schr\"odinger initial-value problems of the form
\begin{equation}\label{GNLS}
(i\partial_t+\Delta_\g)u=\mathcal{N}(u),\qquad u(0)=\phi,
\end{equation}
on certain noncompact Riemannian manifolds $(M,{\bf g})$. 
Here $\Delta_\g=g^{ij}\left(\partial_{ij}-\Gamma_{ij}^k\partial_k\right)$ is the (negative) Laplace-Beltrami operator of $(M,{\bf g})$.
In Euclidean spaces, the subcritical theory of such nonlinear Schr\"{o}dinger equations is well established, see for example the books \cite{Cazenave:book} or \cite{Tao:book} for many references. Many of the subcritical methods extend also to the study of critical equations with small data. The case of large-data critical Schr\"{o}dinger equations is more delicate, and was first considered by Bourgain \cite{B} and Grillakis \cite{G} for defocusing Schr\"{o}dinger equations with pure power nonlinearities and spherically symmetric data. The spherical symmetry assumption was removed, in dimension $d=3$ by Colliander-Keel-Staffilani-Takaoka-Tao \cite{CKSTTcrit} global well-posedness was then extended to higher dimensions $d\geq 4$ by Ryckman-Visan \cite{RV} and Visan \cite{V}.

A key development in the theory of large-data critical dispersive problems was the work of Kenig-Merle \cite{KeMe}, on spherically symmetric solutions of the energy-critical focusing NLS in $\mathbb{R}^3$. The methods developed in this paper found applications in many other large-data critical dispersive problems, leading to complete solutions or partial results. We adapt this point of view in our variable coefficient setting as well.

To keep things as simple as  possible on a technical level, in this paper we consider only the energy-critical defocusing Schr\"{o}dinger equation
\begin{equation}\label{eq1}
(i\partial_t+\Delta_\g)u=u|u|^4
\end{equation}
in the hyperbolic space $\H^3$. Suitable solutions on the time interval $(T_1,T_2)$ of \eqref{eq1} satisfy mass and energy conservation, in the sense that the functions
\begin{equation}\label{conserve}
E^0(u)(t):=\int_{\H^3}|u(t)|^2\,d\mu,\qquad E^1(u)(t):=\frac{1}{2}\int_{\H^3}|\nabla_\g u(t)|^2\,d\mu+\frac{1}{6}\int_{\H^3}|u(t)|^{6}\,d\mu,
\end{equation}
are constant on the interval $(T_1,T_2)$. Our main theorem concerns global well-posedness and scattering in $H^1(\H^3)$ for the initial-value problem associated to the equation \eqref{eq1}. 

\begin{theorem}\label{Main1} (a) (Global well-posedness) If $\phi\in H^1(\H^3)$\footnote{Unlike in Euclidean spaces, in hyperbolic spaces $\H^d$ one has the uniform inequality $\int_{\H^d}|f|^2\,d\mu\lesssim\int_{\H^d}|\nabla f|^2\,d\mu$ for any $f\in C^\infty_0(\H^d)$. In other words $\dot{H}^1(\H^d)\hookrightarrow L^2(\H^d)$.} then there exists a unique global solution $u\in C(\R:H^1(\H^3))$ of the initial-value problem
\begin{equation}\label{eq1.1}
(i\partial_t+\Delta_\g)u=u|u|^{4},\qquad u(0)=\phi.
\end{equation}
In addition, the mapping $\phi\to u$ is a continuous mapping from $H^1(\H^3)$ to $C(\R:H^1(\H^3))$, and the quantities $E^0(u)$ and $E^1(u)$ defined in \eqref{conserve} are conserved.

(b) (Scattering) We have the bound
\begin{equation}\label{sol2}
\|u\|_{L^{10}(\H^3\times\R)}\leq C(\|\phi\|_{H^1(\H^3)}).
\end{equation}
As a consequence, there exist unique $u_{\pm}\in H^1(\H^3)$ such that
\begin{equation}\label{sol4}
\|u(t)-e^{it\Delta_\g}u_{\pm}\|_{H^1(\H^3)}=0\text{ as }t\to\pm\infty.
\end{equation}   
\end{theorem} 

It was observed by Banica \cite{Ba} that the hyperbolic geometry cooperates well with the dispersive nature of Schr\"{o}dinger equations, at least in the case of subcritical problems. In fact the long time dispersion of solutions is stronger in the hyperbolic geometry than in the Euclidean geometry. Intuitively, this is due to the fact that the volume of a ball of radius $R+1$ in hyperbolic spaces is about twice as large as the volume of a ball of radius $R$, if $R\geq 1$; therefore, as outgoing waves advance one unit in the geodesic direction they have about twice as much volume to disperse into. This heuristic can be made precise, see \cite{AnPi,Ba,BaCaDu,BaCaSt,BaDu,Bouc,ChrMar,IoSt,Pie} for theorems concerning subcritical nonlinear Schr\"{o}dinger equations in hyperbolic spaces (or other spaces that interpolate between Euclidean and hyperbolic spaces). The theorems proved in these papers are stronger than the corresponding theorems in Euclidean spaces, in the sense that one obtains better scattering and dispersive properties of the nonlinear solutions.

We remark, however, that the global geometry of the manifold cannot bring any improvements in the case of critical problems. To see this, consider only the case of data of the form
\begin{equation}\label{Alex29}
\phi_N(x)=N^{1/2}\psi(N\Psi^{-1}(x)),
\end{equation}
where $\psi\in C^\infty_0(\mathbb{R}^3)$ and $\Psi:\mathbb{R}^3\to\mathbb{H}^3$ is a suitable local system of coordinates. Assuming that $\psi$ is fixed and letting $N\to\infty$, the functions $\phi_N\in C^\infty_0(\mathbb{H}^3)$ have uniformly bounded $H^1$ norm. For any $T\geq 0$ and $\psi$ fixed, one can prove that the nonlinear solution of \eqref{eq1.1} corresponding to data $\phi_N$ is well approximated by
\begin{equation*}
N^{1/2}v(N\Psi^{-1}(x),N^2t)
\end{equation*}
on the time interval $(-TN^{-2},TN^{-2})$, for $N$ sufficiently large (depending on $T$ and $\psi$), where $v$ is the solution on the time interval $(-T,T)$ of the Euclidean nonlinear Schr\"{o}dinger equation
\begin{equation}\label{Alex30}
(i\partial_t+\Delta)v=v|v|^4,\qquad v(0)=\psi.
\end{equation}
See Section \ref{Eucl} for precise statements. In other words, the solution of the hyperbolic NLS \eqref{eq1.1} with data $\phi_N$ can be regular on the time interval $(-TN^{-2},TN^{-2})$ only if the solution of the Euclidean NLS \eqref{Alex30} is regular on the interval $(-T,T)$. This shows that understanding the Euclidean scale invariant problem is a prerequisite for understanding the problem on any other manifold. Fortunately, we are able to use the main theorem of Colliander-Keel-Staffilani-Takaoka-Tao \cite{CKSTTcrit} as a black box (see the proof of Lemma \ref{step1}).

The previous heuristic shows that understanding the {\it{scaling limit}} problem \eqref{Alex30} is part of understanding the full nonlinear evolution \eqref{eq1.1}, at least if one is looking for uniform  control on all solutions below a certain energy level. This approach was already used in the study of elliptic equations, first in the subcritical case (where the scaling limits are easier) by Gidas-Spruck \cite{GidSpr} and also in the $H^1$ critical setting, see for example Druet-Hebey-Robert \cite{DruHebRob}, Hebey-Vaugon \cite{HebVau}, Schoen \cite{Sch} and (many) references therein for examples. Note however that in the dispersive case, we have to contend with the fact that we are looking at perturbations of a linear operator $i\partial_t+\Delta_\g$ whose kernel is infinite dimensional.

Other critical dispersive models, such as large-data critical wave equations or the Klein--Gordon equation have also been studied extensively, both in the case of the Minkowski space and in other Lorentz manifolds. See for example \cite{BaGe,BaSh,BuLePl,BuPl,Gr1,Gr2,IbMa,IbMaMaNa,IbMaMaNa2,Ka,KeMe2,KilStoVis,La,ShSt1,ShSt2,St} and the book \cite{Tao:book} for further discussion and references. In the case of the wave equation, passing to the variable coefficient setting is somewhat easier due the finite speed of propagation of solutions.

Nonlinear Schr\"{o}dinger equations such as \eqref{GNLS} have also been considered in the setting of compact Riemannian manifolds $(M,\g)$, see \cite{Bo2, Bo1, BuGeTz, BuGeTz2,CKSTTTorus,GerPie}. In this case the conclusions are generally weaker than in Euclidean spaces: there is no scattering to linear solutions, or some other type of asymptotic control of the nonlinear evolution as $t\to\infty$. Moreover, in certain cases such as the spheres $\mathbb{S}^d$, the well-posedness theory requires sufficiently subcritical nonlinearities, due to concentration of certain spherical harmonics. We note however the recent result of Herr-Tataru-Tzvetkov \cite{HeTaTz} on the global well-posedness of the energy critical NLS with small initial data in $H^1(T^3)$.

To simplify the exposition, we use some of the structure of the hyperbolic spaces; in particular we exploit the existence of a large group of isometries that acts transitively on $\mathbb{H}^d$. However the main ingredients in the proof are more basic, and can probably be extended to more general settings. These main ingredients are:

\begin{enumerate}
\item  A dispersive estimate such as \eqref{dispersive}, which gives a good large-data local well-posedness/stability theory (Propositions \ref{localwp} and \ref{stability}).

\item A good Morawetz-type inequality (Proposition \ref{MoraIneq}) to exploit the global defocusing character of the equation.

\item A good understanding of the Euclidean problem, provided in this case by Theorem \ref{MainThmEucl} of Colliander-Keel-Staffilani-Takaoka-Tao \cite{CKSTTcrit}.

\item Some uniform control of the geometry of the manifold at infinity.
\end{enumerate}

The rest of the paper is organized as follows: in Section \ref{preliminaries} we set up the notations, and record the main dispersive estimates on the linear Schr\"{o}dinger flow on hyperbolic spaces. We prove also several lemmas that are used later.

In Section \ref{mainproof} we collect all the necessary ingredients described above, and outline the proof of the main theorem. The only component of the proof that is not known is Proposition \ref{lem4} on the existence of a suitable minimal energy blow-up solution. 

In Section \ref{Eucl} we consider nonlinear solutions of \eqref{eq1.1} corresponding to data that contract at a point, as in \eqref{Alex29}. Using the main theorem in \cite{CKSTTcrit} we prove that such nonlinear solutions extend globally in time and satisfy suitable dispersive bounds.

In Section \ref{profile} we prove our main profile decomposition of $H^1$-bounded sequences of functions in hyperbolic spaces. This is the analogue of Keraani's theorem \cite{Ker} in Euclidean spaces. In hyperbolic spaces we have to distinguish between two types of profiles: Euclidean profiles which may contract at a point, after time and space translations, and hyperbolic profiles which live essentially at frequency\footnote{Here we define the notion of frequency through the Heat kernel, see \eqref{proj}.} $N=1$. The hyperbolic geometry guarantees that profiles of low frequency $N\ll 1$ can be treated as perturbations.
Finally, in Section \ref{proofnew} we use our profile decomposition and orthogonality arguments to complete the proof of Proposition \ref{lem4}.       

\section{Preliminaries}\label{preliminaries}

In this subsection we review some aspects of the harmonic analysis and the geometry of hyperbolic spaces, and summarize our notations. For simplicity, we will use the conventions in \cite{Br}, but one should keep in mind that hyperbolic spaces are the simplest examples of symmetric spaces of the noncompact type, and most of the analysis on hyperbolic spaces can be generalized to this setting (see for example, \cite{He2}).

\subsection{Hyperbolic spaces: Riemannian structure and isometries} For integers $d\geq 2$ we consider the Minkowski space $\R^{d+1}$ with the standard Minkowski metric $-(dx^0)^2+(dx^1)^2+\ldots+(dx^d)^2$ and define the bilinear form on $\R^{d+1}\times\R^{d+1}$,
\begin{equation*}
[x,y]=x^0y^0-x^1y^1-\ldots-x^dy^d.
\end{equation*}
The hyperbolic space $\H^d$ is defined as 
\begin{equation*}
\H^d=\{x\in\R^{d+1}:[x,x]=1\text{ and }x^0>0\}.
\end{equation*}
Let ${\bf{0}}=(1,0,\ldots,0)$ denote the origin of $\H^d$. The Minkowski metric on $\R^{d+1}$ induces a Riemannian metric $\g$ on $\H^d$, with covariant derivative $\D$ and induced measure $d\mu$.

We define ${\mathbb G}:=SO(d,1)=SO_e(d,1)$ as the connected Lie group of $(d+1)\times(d+1)$ matrices that leave the form $[.,.]$ invariant. Clearly, $X\in SO(d,1)$ if and only if
\begin{equation*}
{}^{tr}X\cdot I_{d,1}\cdot X=I_{d,1},\quad\det X=1,\quad X_{00}>0,
\end{equation*}
where $I_{d,1}$ is the diagonal matrix $\mathrm{diag}[-1,1,\ldots,1]$ (since $[x,y]=-{}^tx\cdot I_{d,1}\cdot y$). Let ${\mathbb K}=SO(d)$ denote the subgroup of $SO(d,1)$ that fixes the origin ${\bf{0}}$. Clearly, $SO(d)$ is the compact rotation group acting on the variables $(x^1,\ldots,x^d)$. We define also the commutative subgroup $\mathbb{A}$ of $\mathbb{G}$,
\begin{equation}\label{subgroupa}
\mathbb{A}:=\left\{a_s=\begin{bmatrix}
&\ch s&\sh s&0\\
&\sh s&\ch s&0\\
&0&0&I_{d-1}
\end{bmatrix}:s\in\R\right\},
\end{equation}
and recall the Cartan decomposition
\begin{equation}\label{cartan}
\mathbb{G}=\mathbb{K}\mathbb{A}_+\mathbb{K},\qquad \mathbb{A}_+:=\{a_s:s\in[0,\infty)\}.
\end{equation} 

The semisimple Lie group $\mathbb{G}$ acts transitively on $\H^d$ and the hyperbolic space $\H^d$ can be identified with the homogeneous space ${\mathbb G}/{\mathbb K}=SO(d,1)/SO(d)$. Moreover, for any $h\in SO(d,1)$ the mapping $L_h:\H^d\to\H^d$, $L_h(x)=h\cdot x$, defines an isometry of $\H^d$. Therefore, for any $h\in\mathbb{G}$, we define the isometries
\begin{equation}\label{pi}
\pi_h:L^2(\H^d)\to L^2(\H^d),\qquad\pi_h(f)(x)=f(h^{-1}\cdot x).
\end{equation}

We fix normalized coordinate charts which allow us to pass in a suitable way between functions defined on hyperbolic spaces and functions defined on Euclidean spaces. More precisely, for any $h\in SO(d,1)$ we define the diffeomorphism
\begin{equation}\label{PsiCoord}
\Psi_h:\R^d\to \H^d,\qquad \Psi_h(v^1,\ldots,v^d)=h\cdot(\sqrt{1+|v|^2},v^1,\ldots,v^d).
\end{equation}
Using these diffeomorphisms we define, for any $h\in\mathbb{G}$,
\begin{equation}\label{pitilde}
\widetilde{\pi}_h:C(\R^d)\to C(\H^d),\qquad\widetilde{\pi}_h(f)(x)=f(\Psi_h^{-1}(x)).
\end{equation}
We will use the diffeomorphism $\Psi_I$ as a global coordinate chart on $\mathbb{H}^d$, where $I$ is the identity element of $\mathbb{G}$. We record the integration formula
\begin{equation}\label{changevar}
\int_{\mathbb{H}^d}f(x)\,d\mu(x)=\int_{\mathbb{R}^d}f(\Psi_I(v))(1+|v|^2)^{-1/2}\,dv
\end{equation}
for any $f\in C_0(\mathbb{H}^d)$.

\subsection{The Fourier transform on hyperbolic spaces}
The Fourier transform (as defined by Helgason \cite{He3} in the more general setting of symmetric spaces) takes suitable functions defined on $\H^d$ to functions defined on $\R\times\S^{d-1}$. For $\omega\in \S^{d-1}$ and $\lambda\in\C$, let $b(\omega)=(1,\omega)\in\R^{d+1}$ and 
\begin{equation*}
h_{\lambda,\omega}:\H^d\to\C,\quad h_{\lambda,\omega}(x)=[x,b(\omega)]^{i\lambda-\rho},
\end{equation*}
where
\begin{equation*}
\rho=(d-1)/2.
\end{equation*}
It is known that
\begin{equation}\label{fou}
\Delta_\g h_{\lambda,\omega}=-(\lambda^2+\rho^2)h_{\lambda,\omega},
\end{equation}
where $\Delta_\g$ is the Laplace-Beltrami operator on $\H^d$. The Fourier transform of $f\in C_0(\H^d)$ is defined by the formula
\begin{equation}\label{fht}
\widetilde{f}(\lambda,\omega)=\int_{\H^d}f(x)h_{\lambda,\omega}(x)\,d\mu=\int_{\H^d}f(x)[x,b(\omega)]^{i\lambda-\rho}\,d\mu.
\end{equation}
This transformation admits a Fourier inversion formula: if $f\in C^\infty_0(\H^d)$ then
\begin{equation}\label{Finv}
f(x)=\int_0^\infty\int_{\S^{d-1}}\widetilde{f}(\lambda,\omega)[x,b(\omega)]^{-i\lambda-\rho}|\c(\lambda)|^{-2}\,d\lambda d\omega,
\end{equation}
where, for a suitable constant $C$,
\begin{equation*}
\c(\lambda)=C\frac{\Gamma(i\lambda)}{\Gamma(\rho+i\lambda)}
\end{equation*}
is the Harish-Chandra $\c$-function corresponding to $\H^d$, and the invariant measure of $\S^{d-1}$ is normalized to $1$. It follows from \eqref{fou} that
\begin{equation}\label{fou2}
\widetilde{\Delta_\g f}(\lambda,\omega)=-(\lambda^2+\rho^2)\widetilde{f}(\lambda,\omega).
\end{equation}
We record also the nontrivial identity
\begin{equation*}
\int_{\S^{d-1}}\widetilde{f}(\lambda,\omega)[x,b(\omega)]^{-i\lambda-\rho}d\omega=\int_{\S^{d-1}}\widetilde{f}(-\lambda,\omega)[x,b(\omega)]^{i\lambda-\rho}d\omega
\end{equation*}
for any $f\in C^\infty_0(\H^d)$, $\lambda\in \C$, and $x\in\H^d$.

According to the Plancherel theorem, the Fourier transform $f\to\widetilde{f}$ extends to an isometry of $L^2(\H^d)$ onto $L^2(\R_+\times\S^{d-1},|\c(\lambda)|^{-2}d\lambda d\omega)$; moreover
\begin{equation}\label{Plancherel}
\int_{\H^d}f_1(x)\overline{f_2(x)}\,d\mu=\frac{1}{2}\int_{\R\times\S^{d-1}}\widetilde{f_1}(\lambda,\omega)\overline{\widetilde{f_2}(\lambda,\omega)}|\c(\lambda)|^{-2}\,d\lambda d\omega,
\end{equation}
for any $f_1,f_2\in L^2(\H^d)$. As a consequence, any bounded multiplier $m:\R_+\to\mathbb{C}$ defines a bounded operator $T_m$ on $L^2(\H^d)$ by the formula
\begin{equation}\label{defop}
\widetilde{T_m(f)}(\lambda,\omega)=m(\lambda)\cdot \widetilde{f}(\lambda,\omega).
\end{equation}

The question of $L^p$ boundedness of operators defined by multipliers as in \eqref{defop} is more delicate if $p\neq 2$. A necessary condition for boundedness on $L^p(\H^d)$ of the operator $T_m$ is that the multiplier $m$ extend to an even analytic function in the interior of the region $\mathcal{T}_p=\{\lambda\in\C:|\Im\lambda|<|2/p-1|\rho\}$ (see \cite{ClSt}). Conversely, if $p\in(1,\infty)$ and $m:\mathcal{T}_p\to\mathbb{C}$ is an even analytic function which satisfies the symbol-type bounds
\begin{equation}\label{difeq}
|\partial^\alpha m(\lambda)|\leq C(1+|\lambda|)^{-\alpha}\text{ for any }\al\in[0,d+2]\cap\mathbb{Z}\text{ and }\lambda\in\mathcal{T}_p,
\end{equation}  
then $T_m$ extends to a bounded operator on $L^p(\H^d)$ (see \cite{StTo}).

As in Euclidean spaces, there is a connection between convolution operators in hyperbolic spaces and multiplication operators in the Fourier space. To state this connection precisely, we normalize first the Haar measures on ${\mathbb K}$ and ${\mathbb G}$ such that $\int_{\mathbb K} 1\,dk=1$ and 
\begin{equation*}
\int_{\mathbb G}f(g\cdot {\bf{0}})\,dg=\int_{\H^d}f(x)\,d\mu
\end{equation*}
for any $f\in C_0(\H^d)$. Given two functions $f_1,f_2\in C_0({\mathbb G})$ we define the convolution
\begin{equation}\label{convo}
(f_1\ast f_2)(h)=\int_{\mathbb G}f_1(g)f_2(g^{-1}h)\,dg.
\end{equation}
A function $K:{\mathbb G}\to\mathbb{C}$ is called ${\mathbb K}$-biinvariant if 
\begin{equation}\label{inv1}
K(k_1gk_2)=K(g)\text{ for any }k_1,k_2\in\mathbb{K}.
\end{equation}
Similarly, a function $K:\H^d\to\mathbb{C}$ is called ${\mathbb K}$-invariant (or radial) if 
\begin{equation}\label{inv2}
K(k\cdot x)=K(x)\text{ for any }k\in\mathbb{K}\text{ and }x\in\H^d.
\end{equation}
If $f,K\in C_0(\H^d)$ and $K$ is $\mathbb{K}$-invariant then we define (compare to \eqref{convo})
\begin{equation}\label{convo2}
(f\ast K)(x)=\int_{\mathbb G}f(g\cdot{\bf{0}})K(g^{-1}\cdot x)\,dg.
\end{equation}

If $K$ is $\mathbb{K}$-invariant then the Fourier transform formula \eqref{fht} becomes
\begin{equation}\label{fhtrad}
\widetilde{K}(\lambda,\omega)=\widetilde{K}(\lambda)=\int_{\H^d}K(x)\Phi_{-\lambda}(x)\,d\mu,
\end{equation}
where 
\begin{equation}\label{elem}
\Phi_\lambda(x)=\int_{\S^{d-1}}[x,b(\omega)]^{-i\lambda-\rho}\,d\omega
\end{equation}
is the elementary spherical function. The Fourier inversion formula \eqref{Finv} becomes
\begin{equation}\label{Finvrad}
K(x)=\int_0^\infty\widetilde{K}(\lambda)\Phi_\lambda(x)|\c(\lambda)|^{-2}\,d\lambda,
\end{equation}
for any $\mathbb{K}$-invariant function $K\in C^\infty_0(\H^d)$. With the convolution defined as in \eqref{convo2}, we have the important identity
\begin{equation}\label{keyiden}
\widetilde{(f\ast K)}(\lambda,\omega)=\widetilde{f}(\lambda,\omega)\cdot\widetilde{K}(\lambda)
\end{equation}
for any $f,K\in C_0(\H^d)$, provided that $K$ is $\mathbb{K}$-invariant\footnote{Unlike in Euclidean Fourier analysis, there is no simple identity of this type without the assumption that $K$ is $\mathbb{K}$-invariant.}. 

We define now the inhomogeneous Sobolev spaces on $\H^d$. There are two possible definitions: using the Riemannian structure $\g$ or using the Fourier transform. These two definitions agree. In view of \eqref{fou2}, for $s\in\mathbb{C}$ we define the operator $(-\Delta)^{s/2}$ as given by the Fourier multiplier $\lambda\to (\lambda^2+\rho^2)^{s/2}$. For $p\in(1,\infty)$ and $s\in\R$ we define the Sobolev space $W^{p,s}(\H^d)$ as the closure of $C^\infty_0(\H^d)$ under the norm
\begin{equation*}
\|f\|_{W^{p,s}(\H^d)}=\|(-\Delta)^{s/2}f\|_{L^p(\H^d)}.
\end{equation*}
For $s\in \R$ let $H^s=W^{2,s}$. This definition is equivalent to the usual definition of the Sobolev spaces on Riemannian manifolds (this is a consequence of the fact that the operator $(-\Delta_\g)^{s/2}$ is  bounded on $L^p(\H^d)$ for any $s\in\C$, $\Re s\leq 0$, since its symbol satisfies the differential inequalities \eqref{difeq}). In particular, for $s=1$ and $p\in(1,\infty)$
\begin{equation}\label{sobiden}
\|f\|_{W^{p,1}(\H^d)}=\|(-\Delta)^{1/2}f\|_{L^p(\H^d)}\approx_p\Big[\int_{\H^d}|\nabla_\g f|^{p}\,d\mu\Big]^{1/p},
\end{equation}
where
\begin{equation*}
|\nabla_\g f|:=|\D^\al f\D_\al\overline{f}|^{1/2}.
\end{equation*}
We record also the Sobolev embedding theorem
\begin{equation}\label{Sobemb}
W^{p,s}\hookrightarrow L^q\qquad\text{ if }1<p\leq q<\infty\text{ and }s=d/p-d/q.
\end{equation}      

\subsection{Dispersive estimates} Most of our perturbative analysis in the paper is based on the Strichartz estimates for the linear Schr\"{o}dinger flow. For any $\phi\in H^s(\H^d)$, $s\in\R$, let $e^{it\Delta_\g}\phi\in C(\R:H^s(\H^d))$ denote the solution of the free Schr\"{o}dinger evolution with data $\phi$, i.e.
\begin{equation*}
\widetilde{e^{it\Delta_\g}\phi}(\lambda,\omega)=\widetilde{\phi}(\lambda,\omega)\cdot e^{-it(\lambda^2+\rho^2)}.
\end{equation*}
The main inequality we need is the dispersive estimate{\footnote{In fact this estimate can be improved if $|t|\geq 1$, see \cite[Lemma 3.3]{IoSt}. This leads to better control of the longtime behaviour of solutions of subcritical Schr\"{o}dinger equations in hyperbolic spaces, compared to the behaviour of solutions of the same equations in Euclidean spaces (see \cite{Ba}, \cite{BaCaSt}, \cite{IoSt}, and \cite{AnPi}).} (see \cite{AnPi,Ba,BaCaSt,IoSt,Pie})
\begin{equation}\label{dispersive}
\|e^{it\Delta_\g}\|_{L^p\to L^{p'}}\lesssim |t|^{-d(1/p-1/2)},\qquad p\in[2d/(d+2),2],\,\,p'=p/(p-1),
\end{equation}
for any $t\in\mathbb{R}\setminus\{0\}$. The Strichartz estimates below then follow from the general theorem of Keel-Tao \cite{KeTa}.  

\begin{proposition}\label{StricEst} (Strichartz estimates) Assume that $d\geq 3$ and $I=(a,b)\subseteq \R$ is a bounded open interval. 
 
(i) If $\phi\in L^2(\H^d)$ then
\begin{equation}\label{Strichartz1}
\|e^{it\Delta_\g}\phi\|_{(L^\infty_tL^2_x\cap L^2_tL^{2d/(d-2)}_x)(\H^d\times I)}\lesssim\|\phi\|_{L^2}.
\end{equation}

(ii) If $F\in (L^1_tL^2_x+L^2_tL^{2d/(d+2)}_x)(\H^d\times I)$ then
\begin{equation}\label{Strichartz2}
\Big|\Big|\int_{a}^te^{i(t-s)\Delta_\g}F(s)\,ds\Big|\Big|_{(L^\infty_tL^2_x\cap L^2_tL^{2d/(d-2)}_x)(\H^d\times I)}\lesssim\|F\|_{(L^1_tL^2_x+L^2_tL^{2d/(d+2)}_x)(\H^d\times I)}.
\end{equation}
\end{proposition}

To exploit these estimates in dimension $d=3$, for any interval $I\subseteq\R$ and $f\in C(I:H^{-1}(\H^3))$ we define
\begin{equation}\label{mainnorms}
\begin{split}
&\|f\|_{Z(I)}:=\|f\|_{L^{10}_{t,x}(\H^3\times I)},\\
&\|f\|_{S^k(I)}:=\|(-\Delta)^{k/2}f\|_{(L^\infty_tL^2_x\cap L^2_tL^{6}_x)(\H^3\times I)},\qquad k\in[0,\infty)\\
&\|f\|_{N^k(I)}:=\|(-\Delta)^{k/2}f\|_{(L^1_tL^2_x+L^2_tL^{6/5}_x)(\H^3\times I)},\qquad k\in[0,\infty).
\end{split}
\end{equation}
We use the $S^1$ norms to estimate solutions of linear and nonlinear Schr\"{o}dinger equations. Nonlinearities are estimated using the $N^1$ norms. The $L^{10}$ norm is the ``scattering'' norm, which controls the existence of strong solutions of the nonlinear Schr\"{o}dinger equation, see Proposition \ref{localwp} and Proposition \ref{stability} below.

\subsection{Some lemmas}\label{lemmas}

In this subsection we collect and prove several lemmas that will be used later in the paper. For $N>0$ we define the operator $P_N:L^2(\mathbb{H}^3)\to L^2(\mathbb{H}^3)$,
\begin{equation}\label{proj}
\begin{split}
&P_N:=N^{-2}\Delta_\g e^{N^{-2}\Delta_\g},\\
&\widetilde{P_Nf}(\lambda,\omega)=-N^{-2}(\lambda^2+1)e^{-N^{-2}(\lambda^2+1)}\widetilde{f}(\lambda,\omega).
\end{split}
\end{equation}
One should think of $P_N$ as a substitute for the usual Littlewood-Paley projection operator in Euclidean spaces that restricts to frequencies of size $\approx N$; this substitution is necessary in order to have a suitable $L^p$ theory for these operators, since only real-analytic multipliers can define bounded operators on $L^p(\mathbb{H}^3)$ (see \cite{ClSt}). In view of the Fourier inversion formula
\begin{equation}\label{proj2}
\begin{split}
&P_Nf(x)=\int_{\mathbb{H}^3}f(y)P_N(d(x,y))\,d\mu(y),\\
&|P_N(r)|\lesssim N^3(1+Nr)^{-5}e^{-4r}.
\end{split}
\end{equation}

The estimates in the following lemma will be used in Section \ref{profile}. 

\begin{lemma}\label{lem12}
(i) Given $\epsilon\in(0,1]$ there is $R_\epsilon\geq 1$ such that for any $x\in\mathbb{H}^3$, $N\geq 1$, and $f\in H^1(\mathbb{H}^3)$, 
\begin{equation*}
|P_Nf(x)|\lesssim N^{1/2}(\|f\cdot\mathbf{1}_{B(x,R_\epsilon N^{-1})}\|_{L^6(\mathbb{H}^3)}+\epsilon\|f\|_{L^6(\mathbb{H}^3)})
\end{equation*}
where $B(x,r)$ denotes the ball $B(x,r)=\{y\in\mathbb{H}^3:d(x,y)<r\}$.

(ii) For any $f\in H^1(\mathbb{H}^3)$,
\begin{equation*}
\|f\|_{L^6(\mathbb{H}^3)}\lesssim\|\nabla f\|^{1/3}_{L^2(\mathbb{H}^3)}\cdot\sup_{N\geq 1,\,x\in\mathbb{H}^3}\big[N^{-1/2}|P_Nf(x)|]^{2/3}.
\end{equation*}
\end{lemma}

\begin{proof}[Proof of Lemma \ref{lem12}] (i) The inequality follows directly from \eqref{proj2}:
\begin{equation*}
\begin{split}
|P_Nf(x)|&\lesssim \int_{B(x,R_\epsilon N^{-1})}|f(y)|\,|P_N(d(x,y))|\,d\mu(y)+\int_{{}^cB(x,R_\epsilon N^{-1})}|f(y)|\,|P_N(d(x,y))|\,d\mu(y)\\
&\lesssim \|f\cdot\mathbf{1}_{B(x,R_\epsilon N^{-1})}\|_{L^6(\mathbb{H}^3)}\cdot A_{N,0,6/5}+\|f\|_{L^6(\mathbb{H}^3)}\cdot A_{N,R_\eps,6/5},
\end{split}
\end{equation*}
where, for $R\in[0,\infty)$, $N\in[1,\infty)$ and $p\in[1,2]$
\begin{equation*}
\begin{split}
A_{N,R,p}:&=\Big[\int_{d(\mathbf{0},y)\geq RN^{-1}}|P_N(d(0,y))|^p\,d\mu(y)\Big]^{1/p}\lesssim \Big[\int_{RN^{-1}}^\infty|P_N(r)|^{p}(\sh r)^2\,dr\Big]^{1/p}\\
&\lesssim N^3\Big[\int_{RN^{-1}}^\infty(1+Nr)^{-5p}r^2\,dr\Big]^{1/p}\lesssim N^{3-3/p}(1+R)^{-1}.
\end{split}
\end{equation*}
The inequality follows if $R_\eps=1/\epsilon$.

(ii) For any $f\in H^1(\mathbb{H}^3)$ we have the identity
\begin{equation}\label{ml9}
f=c\int_{N=0}^\infty N^{-1}P_N(f)\,dN.
\end{equation}
Thus, with $A:=\sup_{N\geq 0}\|N^{-1/2}P_Nf\|_{L^\infty(\mathbb{H}^3)}$
\begin{equation*}
\begin{split}
\int_{\mathbb{H}^3}|f|^6\,d\mu&\lesssim \int_{\mathbb{H}^3}\int_{0\leq N_1\leq\ldots\leq N_6}|P_{N_1}f|\cdot\ldots\cdot|P_{N_6}f|\,\frac{dN_1}{N_1}\ldots \frac{dN_6}{N_6}d\mu\\
&\lesssim A^4\int_{\mathbb{H}^3}\int_{0\leq N_5\leq N_6}N_5^2|P_{N_5}f||P_{N_6}f|\,\frac{dN_5}{N_5}\frac{dN_6}{N_6}d\mu\\
&\lesssim A^4\int_{\mathbb{H}^3}\int_0^\infty N|P_Nf|^2\,dN d\mu.
\end{split}
\end{equation*}
The claim follows since
\begin{equation*}
\int_{\mathbb{H}^3}\int_0^\infty N|P_Nf|^2\,dN d\mu=c\|(-\Delta)^{1/2}f\|_{L^2(\mathbb{H}^3)}^2,
\end{equation*}
as a consequence of the Plancherel theorem and the definition of the operators $P_N$, and, for any $N\in[0,1)$,
\begin{equation}\label{ml0}
\|N^{-1/2}P_Nf\|_{L^\infty(\mathbb{H}^3)}\lesssim \|P_2f\|_{L^\infty(\mathbb{H}^3)}.
\end{equation}
\end{proof}

We will also need the following technical estimate:

\begin{lemma}\label{locsmo}
Assume $\psi\in H^1(\mathbb{H}^3)$ satisfies
\begin{equation}\label{ml1}
\|\psi\|_{H^1(\mathbb{H}^3)}\leq 1,\qquad \sup_{K\geq 1,\,t\in\mathbb{R},\,x\in\mathbb{H}^3}K^{-1/2}|P_Ke^{it\Delta_\g}\psi(x)|\leq\delta,
\end{equation}
for some $\delta\in (0,1]$. Then, for any $R>0$ there is $C(R)\geq 1$ such that 
\begin{equation}\label{ml2}
N^{1/2}\|\nabla_\g e^{it\Delta_\g}\psi\|_{L^5_tL^{15/8}_x(B(x_0,RN^{-1})\times(t_0-R^2N^{-2},t_0+R^2N^{-2}))}\leq C(R)\delta^{1/20}
\end{equation}
for any $N\geq 1$, any $t_0\in\mathbb{R}$, and any $x_0\in\mathbb{H}^3$.
\end{lemma} 

\begin{proof}[Proof of Lemma \ref{locsmo}] We may assume $R=1$, $x_0=\mathbf{0}$, $t_0=0$. It follows from \eqref{ml1} that for any $K>0$ and $t\in\mathbb{R}$
\begin{equation*}
\|P_Ke^{it\Delta_\g}\psi\|_{L^\infty(\mathbb{H}^3)}\lesssim \delta K^{1/2},\qquad \|P_Ke^{it\Delta_\g}\psi\|_{L^6(\mathbb{H}^3)}\lesssim 1,
\end{equation*}
therefore, by interpolation,
\begin{equation*}
\|P_Ke^{it\Delta_\g}\psi\|_{L^{12}(\mathbb{H}^3)}\lesssim \delta^{1/2} K^{1/4}.
\end{equation*}
Thus, for any $K>0$ and $t\in\mathbb{R}$
\begin{equation*}
\|\nabla_\g(P_Ke^{it\Delta_\g}\psi)\|_{L^{12}(\mathbb{H}^3)}\lesssim \delta^{1/2} K^{1/4}(K+1), 
\end{equation*}
which shows that, for any $K>0$ and $N\geq 1$,
\begin{equation}\label{ml3}
N^{1/2}\|\nabla_\g (P_Ke^{it\Delta_\g}\psi)\|_{L^5_tL^{15/8}_x(B({\bf{0}},N^{-1})\times(-N^{-2},N^{-2}))}\lesssim \delta^{1/2} K^{1/4}(K+1)N^{-5/4}.
\end{equation}

We will prove below that for any $N\geq 1$ and $K\geq N$
\begin{equation}\label{ml4}
\|\nabla_\g (P_Ke^{it\Delta_\g}\psi)\|_{L^2_{x,t}(B({\bf{0}},N^{-1})\times(-N^{-2},N^{-2}))}\lesssim (NK)^{-1/2}.
\end{equation}
Assuming this and using the energy estimate 
\begin{equation*}
\|\nabla_\g (P_Ke^{it\Delta_\g}\psi)\|_{L^\infty_tL^2_x(\mathbb{H}^3\times\mathbb{R})}\lesssim 1,
\end{equation*}
we have, by interpolation,
\begin{equation*}
\|\nabla_\g (P_Ke^{it\Delta_\g}\psi)\|_{L^5_tL^2_x(B({\bf{0}},N^{-1})\times(-N^{-2},N^{-2}))}\lesssim (NK)^{-1/5}.
\end{equation*}
Therefore, for any $N\geq 1$ and $K\geq N$
\begin{equation}\label{ml5}
N^{1/2}\|\nabla_\g (P_Ke^{it\Delta_\g}\psi)\|_{L^5_tL^{15/8}_x(B({\bf{0}},N^{-1})\times(-N^{-2},N^{-2}))}\lesssim N^{1/5}K^{-1/5}.
\end{equation}
The desired bound \eqref{ml2} follows from \eqref{ml3}, \eqref{ml5}, and the identity \eqref{ml9}.

It remains to prove the local smoothing bound \eqref{ml4}. Many such estimates are known in more general settings, see for example \cite{Do}. We provide below a simple self-contained proof specialized to our case. Assuming $N\geq 1$ fixed, we will construct a real-valued function $a=a_N\in C^\infty(\mathbb{H}^3)$ with the properties
\begin{equation}\label{aprop}
\begin{split}
&|\D^\al a\D_\al a|\lesssim 1\qquad\text{ in }\mathbb{H}^3,\\
&|\Delta_\g(\Delta_\g a)|\lesssim N^3\qquad\text{ in }\mathbb{H}^3,\\
&X^\al X_\al\cdot N\mathbf{1}_{B({\bf{0}},N^{-1})}\lesssim X^\al X^\be\D_\al\D_\be a\qquad\text{ in }\mathbb{H}^3\text{ for any vector-field }X\in T(\mathbb{H}^3).
\end{split}
\end{equation}
Assuming such a function is constructed, we define the Morawetz action
\begin{equation*}
M_a(t)=2\Im\int_{\mathbb{H}^3}\D^\al a(x)\cdot \overline{u}(x)\D_\al u(x)\,d\mu(x),
\end{equation*}
where $u:=P_Ke^{it\Delta_\g}\psi$. A formal computation (see \cite[Proposition 4.1]{IoSt} for a complete justification) shows that
\begin{equation*}
\partial_tM_a(t)=4\Re\int_{\mathbb{H}^3}\D^\al\D^\be a\cdot\D_\al u\D_\be\overline{u}\,d\mu-\int_{\mathbb{H}^3}\Delta_\g(\Delta_\g a)\cdot |u|^2\,d\mu.
\end{equation*}
Therefore, by integrating on the time interval $[-N^{-2},N^{-2}]$ and using the first two properties in \eqref{aprop},
\begin{equation*}
\begin{split}
4\int_{-N^{-2}}^{N^{-2}}&\int_{\mathbb{H}^3}\Re(\D^\al\D^\be a\cdot\D_\al u\D_\be\overline{u})\,d\mu dt\\
&\leq 2\sup_{t\in[-N^{-2},N^{-2}]}|M_a(t)|+\int_{-N^{-2}}^{N^{-2}}\int_{\mathbb{H}^3}|\Delta_\g(\Delta_\g a)|\cdot |u|^2\,d\mu dt\\
&\lesssim \sup_{t\in[-N^{-2},N^{-2}]}\|u(t)\|_{L^2(\mathbb{H}^3)}\|u(t)\|_{H^1(\mathbb{H}^3)}+N^3\int_{-N^{-2}}^{N^{-2}}\|u(t)\|_{L^2(\mathbb{H}^3)}^2\,dt\\
&\lesssim K^{-1}+NK^{-2}.
\end{split}
\end{equation*}
The desired bound \eqref{ml4} follows, in view of the inequality in the last line of \eqref{aprop} and the assumption $K\geq N$ since $a$ is real valued.

Finally, it remains to construct a real-valued function $a\in C^\infty(\mathbb{H}^3)$ satisfying \eqref{aprop}. We are looking for a function of the form
\begin{equation}\label{aprop2}
a(x):=\widetilde{a}(\ch r(x)),\qquad r=d(\mathbf{0},x),\qquad \widetilde{a}\in C^\infty([1,\infty)).
\end{equation}
To prove the inequalities in \eqref{aprop} it is convenient to use coordinates induced by the Iwasawa decomposition of the group $\mathbb{G}$: we define the global diffeomorphism
\begin{equation*}
\Phi:\mathbb{R}^2\times\mathbb{R}\to\mathbb{H}^3,\qquad \Phi(v^1,v^2,s)={}^{tr}(\ch s +e^{-s}|v|^2/2,\sh s+e^{-s}|v|^2/2,e^{-s}v^1,e^{-s}v^2),
\end{equation*}
and fix the global orthonormal frame
\begin{equation*}
e_3:=\partial_s,\quad e_1:=e^s\partial_{v^1},\quad e_2:=e^s\partial_{v^2}.
\end{equation*}
With respect to this frame, the covariant derivatives are
\begin{equation*}
\D_{e_\al}e_\be=\delta_{\al\be}e_3,\,\D_{e_\al}e_3=-e_\al,\,\D_{e_3}e_\al=\D_{e_3}e_3=0,\qquad\text{ for }\al,\be=1,2.
\end{equation*}
See \cite[Section 2]{IoSt} for these calculations. In this system of coordinates we have
\begin{equation}\label{aprop3}
\ch r=\ch s+e^{-s}|v|^2/2.
\end{equation}
Therefore, for $a$ as in \eqref{aprop2}, we have
\begin{equation*}
\D_3a=(\sh s-e^{-s}|v|^2/2)\cdot\widetilde{a}'(\ch r),\quad \D_1a=v^1\cdot\widetilde{a}'(\ch r),\quad \D_2a=v^2\cdot\widetilde{a}'(\ch r). 
\end{equation*}
Using the formula
\begin{equation*}
\D_\al\D_\be a=e_\al(e_\be(a))-(\D_{e_\al}e_\be)(a),\qquad\al,\be=1,2,3.
\end{equation*}
we compute the Hessian
\begin{equation*}
\begin{split}
&\D_1\D_1 a=(v^1)^2\widetilde{a}''(\ch r)+\ch r\widetilde{a}'(\ch r),\,\,\D_2\D_2 a=(v^2)^2\widetilde{a}''(\ch r)+\ch r\widetilde{a}'(\ch r),\\
&\D_1\D_2 a=\D_2\D_1 a=v^1v^2\widetilde{a}''(\ch r),\,\,\D_3\D_3 f=(\sh s-e^{-s}|v|^2/2)^2\widetilde{a}''(\ch r)+\ch r\widetilde{a}'(\ch r),\\
&\D_1\D_3 a=\D_3\D_1 a=v^1(\sh s-e^{-s}|v|^2/2)\widetilde{a}''(\ch r),\\
&\D_2\D_3 a=\D_3\D_2 a=v^2(\sh s-e^{-s}|v|^2/2)\widetilde{a}''(\ch r). 
\end{split}
\end{equation*}
Therefore, using again \eqref{aprop3}
\begin{equation}\label{aprop5}
\D^\al a\D_\al a=(\sh r)^2(\widetilde{a}'(\ch r))^2,\,\,\Delta_\g a=((\ch r)^2-1)\widetilde{a}''(\ch r)+3(\ch r)\widetilde{a}'(\ch r),
\end{equation}
and
\begin{equation}\label{aprop6}
X^\al X^\be\D_\al\D_\be a =\ch r\widetilde{a}'(\ch r)|X|^2+\widetilde{a}''(\ch r)(X^1v^1+X^2v^2+X^3(\sh s-e^{-s}|v|^2/2))^2.
\end{equation}

We fix now $\widetilde{a}$ such that
\begin{equation*}
\widetilde{a}'(y):=(y^2-1+N^{-2})^{-1/2}, \qquad y\in[1,\infty).
\end{equation*}
The first identity in \eqref{aprop} follows easily from \eqref{aprop5}. To prove the second identity in \eqref{aprop}, we use again \eqref{aprop5} to derive
\begin{equation*}
\Delta_\g a=b(\ch r)\quad \text{ where }\quad b(y)=3y(y^2-1+N^{-2})^{-1/2}-y(y^2-1)(y^2-1+N^{-2})^{-3/2}.
\end{equation*}
Using \eqref{aprop5} again, it follows that
\begin{equation*}
|\Delta_\g(\Delta_\g a)|\lesssim y^2(y^2-1+N^{-2})^{-3/2}\quad\text{ where }\quad y=\ch r,
\end{equation*}
which proves the second inequality in \eqref{aprop}. Finally, using \eqref{aprop6},
\begin{equation*}
\begin{split}
X^\al X^\be\D_\al\D_\be a&\geq \ch r\widetilde{a}'(\ch r)|X|^2-((\ch r)^2-1)|\widetilde{a}''(\ch r)|\,|X|^2\\
&=N^{-2}\ch r((\ch r)^2-1+N^{-2})^{-3/2}|X|^2,
\end{split}
\end{equation*}
which proves the last inequality in \eqref{aprop}. This completes the proof of the lemma.
\end{proof}

\section{Proof of the main theorem}\label{mainproof}

In this section we outline the proof of Theorem \ref{Main1}. The main ingredients are a local well-posedness/stability theory for the initial-value problem, which in our case relies only on the Strichartz estimates in Proposition \ref{StricEst}, a global Morawetz inequality, which exploits the defocusing nature of the problem, and a compactness argument, which depends on the Euclidean analogue of Theorem \ref{Main1} proved in \cite{CKSTTcrit}. 

We start with the local well-posedness theory. Let
\begin{equation*}
\mathcal{P}=\{(I,u):I\subseteq\mathbb{R}\text{ is an open interval and }u\in C(I:H^1(\H^3))\}
\end{equation*}
with the natural partial order
\begin{equation*}
(I,u)\leq (I',u')\text{ if and only if }I\subseteq I'\text{ and }u'(t)=u(t)\text{ for any }t\in I.
\end{equation*}

\begin{proposition}\label{localwp}
(Local well-posedness) Assume $\phi\in H^1(\H^3)$. Then there is a unique maximal solution $(I,u)=(I(\phi),u(\phi))\in \mathcal{P}$, $0\in I$, of the initial-value problem
\begin{equation}\label{NLSivp}
(i\partial_t+\Delta_\g)u=u|u|^{4},\qquad u(0)=\phi
\end{equation}
on $\H^3\times I$. In addition $\|u\|_{S^1(J)}<\infty$ for any compact interval $J\subseteq I$, the mass $E^0(u)$ and the energy $E^1(u)$ defined in \eqref{conserve} are constant on $I$, and
\begin{equation}\label{blowup}
\begin{split}
&\text{ if }I_+:=I\cap[0,\infty)\text{ is bounded then }\|u\|_{Z(I_+)}=\infty,\\
&\text{ if }I_-:=I\cap(-\infty,0]\text{ is bounded then }\|u\|_{Z(I_-)}=\infty.
\end{split}
\end{equation}
\end{proposition}

In other words, local-in-time solutions of the equation exist and extend as strong solutions as long as their spacetime $L^{10}_{x,t}$ norm does not blow up. We complement this with a stability result. 

\begin{proposition}\label{stability}
(Stability) Assume $I$ is an open interval, $\rho\in[-1,1]$, and $\widetilde{u}\in C(I:H^1(\H^3))$ satisfies the approximate  Schr\"{o}dinger equation
\begin{equation*}
(i\partial_t+\Delta_\g)\widetilde{u}=\rho\widetilde{u}|\widetilde{u}|^4+e\quad\text{ on }\mathbb{H}^3\times I.
\end{equation*}
Assume in addition that
\begin{equation}\label{ume}
\|\widetilde{u}\|_{L^{10}_{t,x}(\mathbb{H}^3\times I)}+\sup_{t\in I}\|\widetilde{u}(t)\|_{H^1(\H^3)}\leq M,
\end{equation}
for some $M\in[1,\infty)$. Assume $t_0 \in I$ and $u(t_0)\in H^1(\H^3)$ is such that the smallness condition
\begin{equation}\label{safetycheck}
\|u(t_0) - \widetilde{u}(t_0)\|_{H^1(\H^3)}+\| e \|_{N^1(I)}\leq \eps
\end{equation}
holds for some $0 < \eps < \eps_1$, where $\eps_1\leq 1$ is a small constant $\eps_1 = \eps_1(M) > 0$.

Then there exists a solution $u\in C(I:H^1(\mathbb{H}^3))$ of the Schr\"{o}dinger equation 
\begin{equation*}
(i\partial_t+\Delta_\g)u=\rho u|u|^4\text{ on }\mathbb{H}^3\times I,
\end{equation*}
and
\begin{equation}\label{output}
\begin{split}
\| u \|_{S^1(\mathbb{H}^3\times I)}+\|\widetilde{u}\|_{S^1(\mathbb{H}^3\times I)}&\leq C(M),\\
\| u - \widetilde u \|_{S^1(\mathbb{H}^3\times I)}&\leq C(M)\eps.
 \end{split}
\end{equation}
\end{proposition}

Both Proposition \ref{localwp} and Proposition \ref{stability} are standard consequences of the Strichartz estimates and Sobolev embedding theorem \eqref{Sobemb}, see for example \cite[Section 3]{CKSTTcrit}. We will use Proposition \ref{stability} with $\rho=0$ and with $\rho=1$ to estimate linear and nonlinear solutions on hyperbolic spaces. 

We need also the global Morawetz estimate proved in \cite[Proposition 4.1]{IoSt}.    

\begin{proposition}\label{MoraIneq}
Assume that $I\subseteq\R$ is an open interval, and $u\in C(I:H^1(\H^3))$ is a solution of the equation
\begin{equation*}
(i\partial_t+\Delta_\g)u=u|u|^4\text{ on }\H^3\times I.
\end{equation*}
Then, for any $t_1,t_2\in I$,
\begin{equation}\label{mor}
\|u\|_{L^6(\H^3\times[t_1,t_2])}^6\lesssim \sup_{t\in[t_1,t_2]}\|u(t)\|_{L^2(\H^3)}\|u(t)\|_{H^1(\H^3)}.
\end{equation}
\end{proposition}

We turn now to the proof of the main theorem. Recall the conserved energy $E^1(u)$ defined in \eqref{conserve}. For any $E\in[0,\infty)$ let $S(E)$ be defined by
\begin{equation*}
S(E)=\sup\{ \Vert u\Vert_{Z(I)},E^1(u)\le E\},
\end{equation*}
where the supremum is taken over all solutions $u\in C(I:H^1(\H^3))$ defined on an interval $I$ and of energy less than $E$. We also define
$$E_{max}=\sup\{E,S(E)<\infty\}.$$
Using Proposition \ref{stability} with $\widetilde{u}\equiv 0,e\equiv 0$, $I=\R$, $M=1$, $\eps\ll 1$, one checks that $E_{\max}>0$. It follows from Proposition \ref{localwp} that if $u$ is a solution of \eqref{eq1} and $E(u)<E_{max}$, then $u$ can be extended to a globally defined solution which scatters.

If $E_{max}=+\infty$, then Theorem \ref{Main1} is proved, as a consequence of Propositions \ref{localwp} and \ref{stability}. If we assume that $E_{max}<+\infty$, then, there exists a sequence of solutions satisfying the hypothesis of the following key proposition.

\begin{proposition}\label{lem4}
Let $u_k\in C((-T_k,T^k):H^1(\H^3))$, $k=1,2,\ldots$, be a sequence of nonlinear solutions of the equation
\begin{equation*}
(i\partial_t+\Delta_\g)u=u|u|^4,
\end{equation*}
defined on open intervals $(-T_k,T^k)$ such that $E(u_k)\to E_{max}$. Let $t_k\in(-T_k,T^k)$ be a sequence of times with
\begin{equation}\label{CondForCompactnessBigSNorm}
\lim_{k\to\infty}\Vert u_k\Vert_{Z(-T_k,t_k)}=\lim_{k\to\infty}\Vert u_k\Vert_{Z(t_k,T^k)}=+\infty.
\end{equation}
Then there exists $w_0\in H^1(\H^3)$ and a sequence of isometries $h_k\in\mathbb{G}$ such that, up to passing to a subsequence, $u_k(t_k,h_k^{-1}\cdot x)\to w_0(x)\in H^1$ strongly.
\end{proposition}

Using these propositions we can now prove our main theorem.

\begin{proof}[Proof of Theorem \ref{Main1}] Assume for contradiction that $E_{max}<+\infty$. Then, we first claim that there exists a solution $u\in C((-T_\ast,T^\ast):H^1)$ of \eqref{eq1} such that 
\begin{equation}\label{uBlowsUp}
E(u)=E_{max}\text{ and }\Vert u\Vert_{Z(-T_\ast,0)}=\Vert u\Vert_{Z(0,T^\ast)}=+\infty.
\end{equation}
Indeed, by hypothesis, there exists a sequence of solutions $u_k$ defined on intervals $I_k=(-T_k,T^k)$ satisfying $E(u_k)\le E_{max}$ and
\begin{equation*}
\Vert u_k\Vert_{Z(I_k)}\to+\infty.
\end{equation*}
But this is exactly the hypothesis of Proposition \ref{lem4}, for suitable points $t_k\in(-T_k,T^k)$. Hence, up to a subsequence, we get that there exists a sequence of isometries $h_k\in\mathbb{G}$ such that $\pi_{h_k}(u_k(t_k))\to w_0$ strongly in $H^1$. Now, let $u\in C((-T_\ast,T^\ast):H^1(\H^3))$ be the maximal solution of \eqref{NLSivp} with initial data $w_0$, in the sense of Proposition \ref{localwp}. By the stability theory Proposition \ref{stability}, we have that, if $\Vert u\Vert_{Z(0,T^\ast)}<+\infty$, then $T^\ast=+\infty$ and $\Vert u_k\Vert_{Z(t_k,+\infty)}\le C(\Vert u\Vert_{Z(0,+\infty)})$ which is impossible. Similarly, we see that $\Vert u\Vert_{Z(-T_\ast,0)}=+\infty$, which completes the proof of \eqref{uBlowsUp}.

\medskip

We now claim that the solution $u$ obtained in the previous step can be extended to a global solution. Indeed, using Proposition \ref{localwp}, it suffices to see that there exists $\delta>0$ such that for all times $t\in(-T_\ast,T^\ast)$,
\begin{equation*}
\Vert u\Vert_{Z((t-\delta,t+\delta)\cap(-T_\ast,T^\ast))}\le 1.
\end{equation*}
If this were not true, there would exist a sequence $\delta_k \to 0 $ and a sequence of times $t_k\in(-T_\ast+\delta_k,T^\ast-\delta_k)$ such that
\begin{equation}\label{hypotu}
\Vert u\Vert_{Z(t_k-\delta_k,t_k+\delta_k)}\ge 1.
\end{equation}
Applying Proposition \ref{lem4} with $u_k=u$, we see that, up to a subsequence, $\pi_{h_k}(u_k(t_k))\to w$ strongly in $H^1$ for some translations $h_k\in\mathbb{G}$. We consider $z$ the maximal nonlinear solution with initial data $w$, then by the local theory Proposition \ref{localwp}, there exists $\delta>0$ such that
$$\Vert z\Vert_{Z(-\delta,\delta)}\leq 1/2.$$
Proposition \ref{stability} gives that $\Vert u\Vert_{Z(t_k-\delta,t_k+\delta)}\le 1/2+o_k(1)$, which again contradicts our hypothesis \eqref{hypotu}.
In other words, we proved that if $E_{max}<\infty$ then there is a global solution $u\in C(\mathbb{R}:H^1)$ of \eqref{eq1} such that
\begin{equation*}
E(u)=E_{max}\text{ and }\Vert u\Vert_{Z(-\infty,0)}=\Vert u\Vert_{Z(0,\infty)}=+\infty.
\end{equation*}

We claim now that there exists $\delta>0$ such that for all times,
\begin{equation}\label{CompactnessClaim}
\Vert u(t)\Vert_{L^6}\ge\delta.
\end{equation}
Indeed, otherwise, we can find a sequence of times $t_k\in(0,\infty)$ such that $u(t_k)\to 0$ in $L^6$. Applying again Proposition \ref{lem4} to this sequence, we see that, up to a subsequence, there exist $h_k\in\mathbb{G}$ such that $\pi_{h_k}(u(t_k))\to w$ in $H^1$ with $w=0$. But this contradicts conservation of energy.

But now we have a contradiction with the Morawetz estimate \eqref{mor}, which shows that $E_{max}=+\infty$ as desired.
\end{proof}

Propositions \ref{localwp} and \ref{stability} are standard consequences of the Strichartz estimates, while Proposition \ref{MoraIneq} was proved in \cite{IoSt}. Therefore it only remains to prove Proposition \ref{lem4}. We collect the main ingredients in the next two sections and complete the proof of Proposition \ref{lem4} in Section \ref{proofnew}.

\section{Euclidean approximations}\label{Eucl}

In this section we prove precise estimates showing how to compare Euclidean and hyperbolic solutions of both linear and nonlinear Schr\"{o}dinger equations. Since the global Euclidean geometry and the global hyperbolic geometry are quite different, such a comparison is meaningful only in the case of rescaled data that concentrate at a point.

We fix a spherically-symmetric function $\eta\in C^\infty_0(\mathbb{R}^3)$ supported in the ball of radius $2$ and equal to $1$ in the ball of radius $1$. Given $\phi\in \dot{H}^1(\mathbb{R}^3)$ and a real number $N\geq 1$ we define
\begin{equation}\label{rescaled}
\begin{split}
Q_N\phi\in C^\infty_0(\mathbb{R}^3),\qquad &(Q_N\phi)(x)=\eta(x/N^{1/2})\cdot(e^{\Delta/N}\phi)(x),\\
\phi_N\in C^\infty_0(\mathbb{R}^3),\qquad &\phi_N(x)=N^{1/2}(Q_N\phi)(Nx),\\
f_{N}\in C^\infty_0(\mathbb{H}^3),\qquad &f_{N}(y)=\phi_N(\Psi_I^{-1}(y)),
\end{split}
\end{equation}
where $\Psi_I$ is defined in \eqref{PsiCoord}. Thus $Q_N\phi$ is a regularized, compactly supported\footnote{This modification is useful to avoid the contribution of $\phi$ coming from the Euclidean infinity, in a uniform way depending on the scale $N$.} modification of the profile $\phi$, $\phi_N$ is an $\dot{H}^1$-invariant rescaling of $Q_N\phi$, and $f_{N}$ is the function obtained by transferring $\phi_N$ to a neighborhood of $\bf{0}$ in $\mathbb{H}^3$. We define also
\begin{equation*}
E^1_{\mathbb{R}^3}(\phi)=\frac{1}{2}\int_{\mathbb{R}^3}|\nabla\phi|^2\,dx+\frac{1}{6}\int_{\mathbb{R}^3}|\phi|^{6}\,dx.
\end{equation*}

We will use the main theorem of \cite{CKSTTcrit}, in the following form.

\begin{theorem}\label{MainThmEucl}
Assume $\psi\in\dot{H}^1(\mathbb{R}^3)$. Then there is a unique global solution $v\in C(\mathbb{R}:\dot{H}^1(\mathbb{R}^3))$ of the initial-value problem
\begin{equation}\label{clo3}
(i\partial_t+\Delta)v=v|v|^4,\qquad v(0)=\psi,
\end{equation}
and
\begin{equation}\label{clo4}
\|\,|\nabla v|\,\|_{L^\infty_tL^2_x\cap L^2_tL^6_x(\mathbb{R}^3\times\mathbb{R})}\leq \widetilde{C}(E^1_{\mathbb{R}^3}(\psi)).
\end{equation}
Moreover this solution scatters in the sense that there exists $\psi^{\pm\infty}\in\dot{H}^1(\mathbb{R}^3)$ such that
\begin{equation}\label{EScat}
\Vert v(t)-e^{it\Delta}\psi^{\pm\infty}\Vert_{\dot{H}^1(\mathbb{R}^3)}\to 0
\end{equation}
as $t\to\pm\infty$.
Besides, if $\psi\in H^5(\mathbb{R}^3)$ then $v\in C(\mathbb{R}:H^5(\mathbb{R}^3))$ and 
\begin{equation*}
\sup_{t\in\mathbb{R}}\|v(t)\|_{H^5(\mathbb{R}^3)}\lesssim_{\|\psi\|_{H^5(\mathbb{R}^3)}}1.
\end{equation*} 
\end{theorem}

The main result in this section is the following lemma:

\begin{lemma}\label{step1}
Assume $\phi\in\dot{H}^1(\mathbb{R}^3)$, $T_0\in(0,\infty)$, and $\rho\in\{0,1\}$ are given, and define $f_{N}$ as in \eqref{rescaled}. Then the following conclusions hold:

(i) There is $N_0=N_0(\phi,T_0)$ sufficiently large such that for any $N\geq N_0$ there is a unique solution $U_{N}\in C((-T_0N^{-2},T_0N^{-2}):H^1(\mathbb{H}^3))$ of the initial-value problem
\begin{equation}\label{clo5}
(i\partial_t+\Delta_\g)U_N=\rho U_N|U_N|^4,\qquad U_N(0)=f_N.
\end{equation}
Moreover, for any $N\geq N_0$,
\begin{equation}\label{clo6}
\|U_N\|_{S^1(-T_0N^{-2},T_0N^{-2})}\lesssim_{E^1_{\mathbb{R}^3}(\phi)}1.
\end{equation}

(ii) Assume $\varepsilon_1\in(0,1]$ is sufficiently small (depending only on $E^1_{\mathbb{R}^3}(\phi)$), $\phi'\in H^5(\mathbb{R}^3)$, and $\|\phi-\phi'\|_{\dot{H}^1(\mathbb{R}^3)}\leq\varepsilon_1$. Let $v'\in C(\mathbb{R}:H^5)$ denote the solution of the initial-value problem
\begin{equation*}
(i\partial_t+\Delta)v'=\rho v'|v'|^4,\qquad v'(0)=\phi'.
\end{equation*}
For $R,N\geq 1$ we define
\begin{equation}\label{clo9}
\begin{split}
&v'_R(x,t)=\eta(x/R)v'(x,t),\qquad\,\,\qquad (x,t)\in\mathbb{R}^3\times(-T_0,T_0),\\
&v'_{R,N}(x,t)=N^{1/2}v'_R(Nx,N^2t),\qquad (x,t)\in\mathbb{R}^3\times(-T_0N^{-2},T_0N^{-2}),\\
&V_{R,N}(y,t)=v'_{R,N}(\Psi_I^{-1}(y),t)\qquad\quad\,\, (y,t)\in\mathbb{H}^3\times(-T_0N^{-2},T_0N^{-2}).
\end{split}
\end{equation}
Then there is $R_0\geq 1$ (depending on $T_0$ and $\phi'$ and $\varepsilon_1$) such that, for any $R\geq R_0$,
\begin{equation}\label{clo18}
\limsup_{N\to\infty}\|U_N-V_{R,N}\|_{S^1(-T_0N^{-2},T_0N^{-2})}\lesssim_{E^1_{\mathbb{R}^3}(\phi)}\varepsilon_1.
\end{equation}
\end{lemma}   

\begin{proof}[Proof of Lemma \ref{step1}] All of the constants in this proof are allowed to depend on $E^1_{\mathbb{R}^3}(\phi)$; for simplicity of notation we will not track this dependence explicitly. Using Theorem \ref{MainThmEucl}
\begin{equation}\label{clo7}
\begin{split}
&\|\nabla v'\|_{(L^\infty_tL^2_x\cap L^2_tL^6_x)(\mathbb{R}^3\times\mathbb{R})}\lesssim 1,\\
&\sup_{t\in\mathbb{R}}\|v'(t)\|_{H^5(\mathbb{R}^3)}\lesssim_{\|\phi'\|_{H^5(\mathbb{R}^3)}}1.
\end{split}
\end{equation}
We will prove that for any $R_0$ sufficiently large there is $N_0$ such that $V_{R_0,N}$ is an almost-solution of the equation \eqref{clo5}, for any $N\geq N_0$. We will then apply Proposition \ref{stability} to upgrade this to an exact solution of the initial-value problem \eqref{clo5} and prove the lemma. 

Let
\begin{equation*}
\begin{split}
e_R(x,t):&=[(i\partial_t+\Delta)v'_R-\rho v'_R|v'_R|^4](x,t)=\rho(\eta(x/R)-\eta(x/R)^5)v'(x,t)|v'(x,t)|^4\\
&+R^{-2}v'(x,t)(\Delta\eta)(x/R)+2R^{-1}\sum_{j=1}^3\partial_jv'(x,t)\partial_j\eta(x/R).
\end{split}
\end{equation*}
Since $|v'(x,t)|\lesssim_{\|\phi'\|_{H^5(\mathbb{R}^3)}}1$, see \eqref{clo7}, it follows that
\begin{equation*}
\sum_{k=1}^3|\partial_ke_R(x,t)|\lesssim_{\|\phi'\|_{H^3(\mathbb{R}^3)}}\mathbf{1}_{[R,2R]}(|x|)\cdot\big[|v'(x,t)|+\sum_{k=1}^3|\partial_kv'(x,t)|+\sum_{k,j=1}^3|\partial_k\partial_jv'(x,t)|\big].
\end{equation*}
Therefore
\begin{equation}\label{clo10}
\lim_{R\to\infty}\|\,|\nabla e_R|\,\|_{L^2_tL^2_x(\mathbb{R}^3\times(-T_0,T_0))}=0.
\end{equation}
Letting
\begin{equation*}
e_{R,N}(x,t):=[(i\partial_t+\Delta)v'_{R,N}-\rho v'_{R,N}|v'_{R,N}|^4](x,t)=N^{5/2}e_R(Nx,N^2t),
\end{equation*}
it follows from \eqref{clo10} that there is $R_0\geq 1$ such that, for any $R\geq R_0$ and $N\geq 1$,
\begin{equation}\label{clo11}
\|\,|\nabla e_{R,N}|\,\|_{L^1_tL^2_x(\mathbb{R}^3\times(-T_0N^{-2},T_0N^{-2}))}\leq\varepsilon_1.
\end{equation}

With $V_{R,N}(y,t)=v'_{R,N}(\Psi_I^{-1}(y),t)$ as in \eqref{clo9}, let
\begin{equation}\label{clo13}
\begin{split}
E_{R,N}(y,t):&=[(i\partial_t+\Delta_\g)V_{R,N}-\rho V_{R,N}|V_{R,N}|^4](y,t)\\
&=e_{R,N}(\Psi_I^{-1}(y),t)+\Delta_gV_{R,N}(y,t)-(\Delta v'_{R,N})(\Psi_I^{-1}(y),t).
\end{split}
\end{equation}
To estimate the difference in the formula above, let $\partial_j$, $j=1,2,3$, denote the standard vector-fields on $\mathbb{R}^3$ and $\widetilde{\partial_j}:=(\Psi_I)_\ast(\partial_j)$ and induced vector-fields on $\mathbb{H}^3$. Using the definition \eqref{PsiCoord} we compute
\begin{equation*}
\g_{ij}(y):=\g_y(\widetilde{\partial_i},\widetilde{\partial_j})=\delta_{ij}-\frac{v_iv_j}{1+|v|^2},\qquad y=\Psi_I(v).
\end{equation*}
Using the standard formula for the Laplace-Beltrami operator in local coordinates 
\begin{equation*}
\Delta_\g f=|\g|^{-1/2}\widetilde{\partial_i}(|\g|^{1/2}\g^{ij}\widetilde{\partial_j}f)
\end{equation*}
we derive the pointwise bound
\begin{equation*}
|\widetilde{\nabla}^1[\Delta_\g f(y)-\Delta(f\circ\Psi_I)(\Psi_I^{-1}(y))]|\lesssim\sum_{k=1}^3|\Psi_I^{-1}(y)|^{k-1}|\widetilde{\nabla}^kf(y)|,
\end{equation*}
for any $C^3$ function $f:\mathbb{H}^3\to\mathbb{C}$ supported in the ball of radius $1$ around $\bf{0}$, where, by definition, for $k=1,2,3$
\begin{equation*}
|\widetilde{\nabla}^k h(y)|:=\sum_{k_1+k_2+k_3=k}|\widetilde{\partial_1}^{k_1}\widetilde{\partial_2}^{k_2}\widetilde{\partial_3}^{k_3}h(y)|.
\end{equation*}
Therefore the identity \eqref{clo13} gives the pointwise bound
\begin{equation*}
\begin{split}
&|\widetilde{\nabla}^1E_{R,N}(y,t)|\\
&\lesssim |\nabla e_{R,N}|(\Psi_I^{-1}(y),t)+\sum_{k=1}^3\sum_{k_1+k_2+k_3=k}|\Psi_I^{-1}(y)|^{k-1}|\partial_1^{k_1}\partial_2^{k_2}\partial_3^{k_3}v'_{R,N}(\Psi_I^{-1}(y),t)|\\
&\lesssim |\nabla e_{R,N}|(\Psi_I^{-1}(y),t)+R^3N^{3/2}\sum_{k_1+k_2+k_3\in\{1,2,3\}}|\partial_1^{k_1}\partial_2^{k_2}\partial_3^{k_3}v'_R(N(\Psi_I^{-1}(y),t)|.
\end{split}
\end{equation*}
Using also \eqref{clo11}, it follows that for any $R_0$ sufficiently large there is $N_0$ such that for any $N\geq N_0$
\begin{equation}\label{clo15}
\|\,|\nabla_\g E_{R_0,N}|\,\|_{L^1_tL^2_x(\mathbb{H}^3\times(-T_0N^{-2},T_0N^{-2}))}\leq 2\varepsilon_1.
\end{equation}

To verify the hypothesis \eqref{ume} of Proposition \ref{stability}, we use \eqref{clo7} and the integral formula \eqref{changevar} to estimate, for $N$ large enough,
\begin{equation}\label{clo16}
\begin{split}
\|V_{R_0,N}&\|_{L^{10}_{x,t}(\mathbb{H}^3\times(-T_0N^{-2},T_0N^{-2}))}+\sup_{t\in(-T_0N^{-2},T_0N^{-2})}\|V_{R_0,N}(t)\|_{H^1(\mathbb{H}^3)}\\
&\lesssim \|v'_{R_0,N}\|_{L^{10}_{x,t}(\mathbb{R}^3\times(-T_0N^{-2},T_0N^{-2}))}+\sup_{t\in(-T_0N^{-2},T_0N^{-2})}\|\nabla v'_{R_0,N}(t)\|_{L^2(\mathbb{R}^3)}\\
&=\|v'_{R_0}\|_{L^{10}_{x,t}(\mathbb{R}^3\times(-T_0,T_0))}+\sup_{t\in(-T_0,T_0)}\|\nabla v'_{R_0}(t)\|_{L^2(\mathbb{R}^3)}\\
&\lesssim 1.
\end{split}
\end{equation}
Finally, to verify the inequality on the first term in \eqref{safetycheck} we estimate, for $R_0,N$ large enough,
\begin{equation}\label{clo17}
\begin{split}
\|f_N-V_{R_0,N}(0)&\|_{H^1(\mathbb{H}^3)}\lesssim \|\phi_N-v'_{R_0,N}(0)\|_{\dot{H}^1(\mathbb{R}^3)}=\|Q_N\phi-v'_{R_0}(0)\|_{\dot{H}^1(\mathbb{R}^3)}\\
&\leq \|Q_N\phi-\phi\|_{\dot{H}^1(\mathbb{R}^3)}+\|\phi-\phi'\|_{\dot{H}^1(\mathbb{R}^3)}+\|\phi'-v'_{R_0}(0)\|_{\dot{H}^1}\leq 3\varepsilon_1.
\end{split}
\end{equation}
The conclusion of the lemma follows from Proposition \ref{stability}, provided that $\varepsilon_1$ is fixed sufficiently small depending on $E^1_{\mathbb{R}^3}(\phi)$. 
\end{proof}

As a consequence, we have the following:

\begin{corollary}\label{step2}
Assume $\psi\in \dot{H}^1(\mathbb{R}^3)$, $\varepsilon>0$, $I\subseteq \mathbb{R}$ is an interval, and
\begin{equation}\label{clo20}
\|\,|\nabla(e^{it\Delta}\psi)|\,\|_{L^{p}_tL^{q}_x(\mathbb{R}^3\times I)}\leq\varepsilon,
\end{equation}
where $2/p+3/q=3/2$, $q\in (2,6]$. For $N\geq 1$ we define, as before,
\begin{equation*}
(Q_N\psi)(x)=\eta(x/N^{1/2})\cdot(e^{\Delta/N}\psi)(x),\,\,\psi_N(x)=N^{1/2}(Q_N\psi)(Nx),\,\,\widetilde{\psi}_N(y)=\psi_N(\Psi_I^{-1}(y)).
\end{equation*}
Then there is $N_1=N_1(\psi,\varepsilon)$ such that, for any $N\geq N_1$,
\begin{equation}\label{clo21}
\|\,|\nabla_\g(e^{it\Delta_\g}\widetilde{\psi}_N)|\,\|_{L^p_tL^q_x(\mathbb{H}^3\times N^{-2}I)}\lesssim_q\varepsilon.
\end{equation}
\end{corollary}

\begin{proof}[Proof of Lemma \ref{step2}] As before, the implicit constants may depend on $E^1_{\mathbb{R}^3}(\psi)$. We may assume that $\psi\in C^\infty_0(\mathbb{R}^3)$. Using the dispersive estimate \eqref{dispersive}, for any $t\neq 0$,
\begin{equation*}
\begin{split}
\|(-\Delta_\g)^{1/2}(e^{it\Delta_\g}\widetilde{\psi}_N)\|_{L^{q}_x(\mathbb{H}^3)}&\lesssim |t|^{3/q-3/2}\|(-\Delta_\g)^{1/2}\widetilde{\psi}_N\|_{L^{q'}_x(\mathbb{H}^3)}\lesssim |t|^{3/q-3/2}\|\,|\nabla\psi_N|\,\|_{L^{q'}_x(\mathbb{R}^3)}\\&\lesssim_\psi|t|^{3/q-3/2}N^{3/q-3/2}.
\end{split}
\end{equation*}
Thus, for $T_1>0$,
\begin{equation*}
\|\,|\nabla_\g(e^{it\Delta_\g}\widetilde{\psi}_N)|\,\|_{L^{p}_tL^{q}_x(\mathbb{H}^3\times[\mathbb{R}\setminus(-T_1N^{-2},T_1N^{-2})])} \lesssim_\psi T_1^{-1/p}.
\end{equation*}
Therefore we can fix $T_1=T_1(\psi,\varepsilon)$ such that, for any $N\geq 1$,
\begin{equation*}
\|\,|\nabla_\g(e^{it\Delta_\g}\widetilde{\psi}_N)|\,\|_{L^{p}_tL^{q}_x(\mathbb{H}^3\times[\mathbb{R}\setminus(-T_1N^{-2},T_1N^{-2})])}\lesssim_q\varepsilon.
\end{equation*}

The desired bound on the remaining interval $N^{-2}I\cap(-T_1N^{-2},T_1N^{-2})$ follows from Lemma \ref{step1} (ii) with $\rho=0$.
\end{proof}

\section{Profile decomposition in hyperbolic spaces}\label{profile}

In this section we show that given a bounded sequence of functions $f_k\in H^1(\mathbb{H}^3)$ we can construct certain {\it{profiles}} and express the functions $f_k$ in terms of these profiles. In other words, we prove the analogue of Keraani's theorem \cite{Ker} in the hyperbolic geometry.

Given $(f,t_0,h_0)\in L^2(\mathbb{H}^3)\times\mathbb{R}\times\mathbb{G}$ we define
\begin{equation}\label{PI}
\Pi_{t_0,h_0}f(x)=(e^{-it_0\Delta_\g}f)(h_0^{-1}x)=(\pi_{h_0}e^{-it_0\Delta_\g}f)(x).
\end{equation}
As in Section \ref{Eucl}, see \eqref{rescaled}, given $\phi\in\dot{H}^1(\mathbb{R}^3)$ and $N\geq 1$, we define
\begin{equation}\label{TN}
T_N\phi(x):=N^{1/2}\widetilde{\phi}(N\Psi_I^{-1}(x))\qquad\text{ where }\qquad\widetilde{\phi}(y):=\eta(y/N^{1/2})\cdot (e^{\Delta/N}\phi)(y),
\end{equation}
and observe that
\begin{equation}\label{TN2}
T_N:\dot{H}^1(\mathbb{R}^3)\to H^1(\mathbb{H}^3)\text{ is a bounded linear operator with }\|T_N\phi\|_{H^1(\mathbb{H}^3)}\lesssim \|\phi\|_{\dot{H}^1(\mathbb{R}^3)}.
\end{equation}
The following is our main definition.

\begin{definition}\label{DefPro}

\begin{enumerate}
\item We define a {\it{frame}} to be a sequence $\mathcal{O}_k=(N_k,t_k,h_k)\in[1,\infty)\times\mathbb{R}\times\mathbb{G}$, $k=1,2,\ldots$, where $N_k\geq 1$ is a scale, $t_k\in\mathbb{R}$ is a time, and $h_k\in\mathbb{G}$ is a translation element. We also assume that either $N_k=1$ for all $k$ (in which case we call $\{\mathcal{O}_k\}_{k\geq 1}$ a hyperbolic frame) or that $N_k\nearrow\infty$ (in which case we call $\{\mathcal{O}_k\}_{k\geq 1}$ a Euclidean frame). Let $\mathcal{F}_e$ denote the set of Euclidean frames,
\begin{equation*}
\mathcal{F}_e=\{\mathcal{O}=\{(N_k,t_k,h_k)\}_{k\geq 1}:\,N_k\in[1,\infty),\,t_k\in\mathbb{R},\,h_k\in\mathbb{G},\,N_k\nearrow\infty\},
\end{equation*}
and let $\mathcal{F}_h$ denote the set of hyperbolic frames, 
\begin{equation*}
\mathcal{F}_h=\{\widetilde{\mathcal{O}}=\{(1,t_k,h_k)\}_{k\geq 1}:t_k\in\mathbb{R},\,h_k\in\mathbb{G}\}.
\end{equation*}

\item We say that two frames $\{(N_k,t_k,h_k)\}_{k\geq 1}$ and $\{(N^\prime_k,t^\prime_k,h^\prime_k)\}_{k\geq 1}$ are {\it{equivalent}} if
\begin{equation}\label{EquivalentFrame}
\limsup_{k\to\infty}\big[\vert\ln(N_k/N_{k'})\vert+N_k^2\vert t_k-t_k^\prime\vert+N_kd(h_k\cdot\mathbf{0},h'_k\cdot\mathbf{0})\big]<+\infty.
\end{equation}
Note that this indeed defines an equivalence relation. Two frames which are not equivalent are called {\it orthogonal}.

\item
Given $\phi\in\dot{H}^1(\mathbb{R}^3)$ and a Euclidean frame $\mathcal{O}=\{\mathcal{O}_k\}_{k\geq 1}=\{(N_k,t_k,h_k)\}_{k\geq 1}\in\mathcal{F}_e$, we define {\it{the Euclidean profile associated with $(\phi,\mathcal{O})$}} as the sequence $\widetilde{\phi}_{\mathcal{O}_k}$, where
\begin{equation}\label{euclprof}
\widetilde{\phi}_{\mathcal{O}_k}:=\Pi_{t_k,h_k}(T_{N_k}\phi),\\
\end{equation}
The operators $\Pi$ and $T$ are defined in \eqref{PI} and \eqref{TN}.

\item Given $\psi\in H^1(\mathbb{H}^3)$ and a hyperbolic frame $\widetilde{\mathcal{O}}=\{\widetilde{\mathcal{O}}_k\}_{k\geq 1}=\{(1,t_k,h_k)\}_{k\geq 1}\in\mathcal{F}_h$ we define {\it{the hyperbolic profile associated with $(\psi,\widetilde{\mathcal{O}})$}} as the sequence $\widetilde{\psi}_{\widetilde{\mathcal{O}}_k}$, where
\begin{equation}\label{hypprof}
\widetilde{\psi}_{\widetilde{\mathcal{O}}_k}:=\Pi_{t_k,h_k}\psi.
\end{equation}
\end{enumerate}
\end{definition}

\begin{definition}\label{DefAbsent}

We say that a sequence $(f_k)_k$ bounded in $H^1(\mathbb{H}^3)$ is {\it absent} from a frame $\mathcal{O}=\{(N_k,t_k,h_k)\}_k$ if its localization to $\mathcal{O}$ converges weakly to $0$, i.e. if for all profiles $\widetilde{\phi}_{\mathcal{O}_k}$ associated to $\mathcal{O}$, there holds that
\begin{equation}\label{WeakConvTo0}
\lim_{k\to\infty}\langle f_k,\widetilde{\phi}_{\mathcal{O}_k}\rangle_{H^1\times H^1(\mathbb{H}^3)}=0.
\end{equation}
\end{definition}

By Lemma \ref{equiv} below, this does not depend on the choice of an equivalent frame. 

\begin{remark}
(i) If $\mathcal{O}=(1,t_k,h_k)_k$ is an hyperbolic frame, this is equivalent to saying that
\begin{equation*}
\Pi_{-t_k,h_k^{-1}}f_k\rightharpoonup 0
\end{equation*}
as $k\to\infty$ in $H^1(\mathbb{H}^3)$.

(ii) If $\mathcal{O}$ is a Euclidean frame, this is equivalent to saying that for all $R>0$
\begin{equation*}
g^R_k(v)=\eta(v/R)N_k^{-1/2}\left(\Pi_{-t_k,h_k^{-1}}f_k\right)(\Psi_I(v/N_k))\rightharpoonup 0
\end{equation*}
as $k\to\infty$ in $\dot{H}^1(\mathbb{R}^3)$.
\end{remark}

We prove first some basic properties of profiles associated to equivalent/orthogonal frames.

\begin{lemma}\label{equiv} (i) Assume $\{\mathcal{O}_k\}_{k\geq 1}=\{(N_k,t_k,h_k)\}_{k\geq 1}$ and $\{\mathcal{O}'_k\}_{k\geq 1}=\{(N'_k,t'_k,h'_k)\}_{k\geq 1}$ are two equivalent Euclidean frames (respectively hyperbolic frames), in the sense of \eqref{EquivalentFrame}, and $\phi\in\dot{H}^1(\mathbb{R}^3)$ (respectively $\phi\in H^1(\mathbb{H}^3)$). Then there is $\phi'\in\dot{H}^1(\mathbb{R}^3)$ (respectively $\phi'\in H^1(\mathbb{H}^3)$) such that, up to a subsequence,
\begin{equation}\label{mb1}
\lim_{k\to\infty}\|\widetilde{\phi}_{\mathcal{O}_k}-\widetilde{\phi'}_{\mathcal{O}'_k}\|_{H^1(\mathbb{H}^3)}=0,
\end{equation}
where $\widetilde{\phi}_{\mathcal{O}_k},\widetilde{\phi'}_{\mathcal{O}'_k}$ are as in Definition \ref{DefPro}.

(ii) Assume $\{\mathcal{O}_k\}_{k\geq 1}=\{(N_k,t_k,h_k)\}_{k\geq 1}$ and $\{\mathcal{O}'_k\}_{k\geq 1}=\{(N'_k,t'_k,h'_k)\}_{k\geq 1}$ are two orthogonal frames (either Euclidean or hyperbolic) and $\widetilde{\phi}_{\mathcal{O}_k},\widetilde{\psi}_{\mathcal{O}'_k}$ are associated profiles. Then, up to a subsequence,
\begin{equation}\label{mb50}
\lim_{k\to\infty}\Big|\int_{\mathbb{H}^3}\D^\al\widetilde{\phi}_{\mathcal{O}_k}\overline{\D_\al\widetilde{\psi}_{\mathcal{O}'_k}}\,d\mu\Big|+\lim_{k\to\infty}\|\widetilde{\phi}_{\mathcal{O}_k}\widetilde{\psi}_{\mathcal{O}'_k}\|_{L^3(\mathbb{H}^3)}=0.
\end{equation}

(iii) If $\widetilde{\phi}_{\mathcal{O}_k}$ and $\widetilde{\psi}_{\mathcal{O}_k}$ are two Euclidean profiles associated to the same frame, then
\begin{equation*}
\begin{split}
\lim_{k\to\infty}\langle\nabla_\g \widetilde{\phi}_{\mathcal{O}_k},\nabla_\g \widetilde{\psi}_{\mathcal{O}_k}\rangle_{L^2\times L^2(\mathbb{H}^3)}&=\lim_{k\to\infty}\int_{\mathbb{H}^3}\D^\al\widetilde{\phi}_{\mathcal{O}_k}\overline{\D_\al\widetilde{\psi}_{\mathcal{O}_k}}\,d\mu\\
&=\int_{\mathbb{R}^3}\nabla\phi(x)\cdot\nabla\overline{\psi}(x)dx=\langle\nabla \phi,\nabla\psi\rangle_{L^2\times L^2(\mathbb{R}^3)}
\end{split}
\end{equation*}
\end{lemma}

\begin{proof}[Proof of Lemma \ref{equiv}] (i) The proof follows from the definitions if $\{\mathcal{O}_k\}_{k\geq 1},\{\mathcal{O}'_k\}_{k\geq 1}$ are hyperbolic frames: by passing to a subsequence we may assume $\lim_{k\to\infty}-t'_k+t_k=\overline{t}$ and $\lim_{k\to\infty}{h'_k}^{-1}h_k=\overline{h}$, and define
\begin{equation*}
\phi':=\Pi_{\overline{t},\overline{h}}\phi.
\end{equation*}

To prove the claim if $\{\mathcal{O}_k\}_{k\geq 1},\{\mathcal{O}'_k\}_{k\geq 1}$ are equivalent Euclidean frames, we decompose first, using the Cartan decomposition \eqref{cartan}
\begin{equation}\label{mb2.5}
{h'_k}^{-1}h_k=m_ka_{s_k}n_k,\qquad m_k,n_k\in\mathbb{K},\,s_k\in[0,\infty).
\end{equation}
Therefore, using the compactness of the subgroup $\mathbb{K}$ and the definition \eqref{EquivalentFrame}, after passing to a subsequence, we may assume that
\begin{equation}\label{mb3}
\begin{split}
&\lim_{k\to\infty}N_k/N'_{k}=\overline{N},\quad \lim_{k\to\infty}N_k^2(t_k-t'_{k})=\overline{t},\\
&\lim_{k\to\infty}m_k=m,\quad\lim_{k\to\infty}n_k=n,\quad\lim_{k\to\infty}N_ks_k=\overline{s}.
\end{split}
\end{equation}
We observe that for any $N\geq 1$, $\psi\in\dot{H}^1(\mathbb{R}^3)$, $t\in\mathbb{R}$, $g\in\mathbb{G}$, and $q\in\mathbb{K}$
\begin{equation*}
\Pi_{t,gq}(T_N\psi)=\Pi_{t,g}(T_N\psi_q)\,\,\text{ where }\,\,\psi_q(x)=\psi(q^{-1}\cdot x).
\end{equation*}
Therefore, in \eqref{mb2.5} we may assume that
\begin{equation*}
m_k=n_k=I,\qquad {h'_k}^{-1}h_k=a_{s_k}.
\end{equation*}

With $\overline{x}=(\overline{s},0,0)$, we define
\begin{equation*}
\phi'(x):=\overline{N}^{1/2}(e^{-i\overline{t}\Delta}\phi)(\overline{N}x-\overline{x}),\qquad\phi'\in\dot{H}^1(\mathbb{R}^3),
\end{equation*}
and define $\widetilde{\phi'}$, $\widetilde{\phi'}_{N'_k}$, and $\widetilde{\phi'}_{\mathcal{O}'_k}$ as in \eqref{euclprof}. The identity \eqref{mb1} is equivalent to
\begin{equation}\label{mb6}
\lim_{k\to\infty}\|T_{N'_k}\phi'-\pi_{{h'_k}^{-1}h_k}e^{i(t'_k-t_k)\Delta_\g}(T_{N_k}\phi)\|_{H^1(\mathbb{H}^3)}=0.
\end{equation}

To prove \eqref{mb6} we may assume that $\phi'\in C^\infty_0(\mathbb{R}^3)$, $\phi\in H^5(\mathbb{R}^3)$, and apply Lemma \ref{step1} (ii) with $\rho=0$. Let $v(x,t)=(e^{it\Delta}\phi)(x)$ and, for $R\geq 1$,
\begin{equation*}
v_R(x,t)=\eta(x/R)v(x,t),\,\,v_{R,N_k}(x,t)=N_k^{1/2}v_R(N_kx,N_k^2t),\,\,V_{R,N_k}(y,t)=v_{R,N_k}(\Psi_I^{-1}(y),t).
\end{equation*}
It follows from Lemma \ref{step1} (ii) that for any $\varepsilon>0$ sufficiently small there is $R_0$ sufficiently large such that, for any $R\geq  R_0$,
\begin{equation}\label{mb7}
\limsup_{k\to\infty}\|e^{i(t'_k-t_k)\Delta_\g}(T_{N_k}\phi)-V_{R,N_k}(t'_k-t_k)\|_{H^1(\mathbb{H}^3)}\leq\varepsilon.
\end{equation}
Therefore, to prove \eqref{mb6} it suffices  to show that, for $R$ large enough,
\begin{equation*}
\limsup_{k\to\infty}\|\pi_{{h_k}^{-1}h'_k}(T_{N'_k}\phi')-V_{R,N_k}(t'_k-t_k)\|_{H^1(\mathbb{H}^3)}\lesssim\varepsilon,
\end{equation*}
which, after examining the definitions and recalling that $\phi'\in C^\infty_0(\mathbb{R}^3)$, is equivalent to
\begin{equation*}
\limsup_{k\to\infty}\|{N'_k}^{1/2}\phi'(N'_k\Psi_I^{-1}({h'_k}^{-1}h_k\cdot y))-N_k^{1/2}v_R(N_k\Psi_I^{-1}(y),N_k^2(t'_k-t_k))\|_{H^1_y(\mathbb{H}^3)}\lesssim\varepsilon.
\end{equation*}
After changing variables $y=\Psi_I(x)$ this is equivalent to
\begin{equation*}
\limsup_{k\to\infty}\|{N'_k}^{1/2}\phi'(N'_k\Psi_I^{-1}({h'_k}^{-1}h_k\cdot \Psi_I(x)))-N_k^{1/2}v_R(N_kx,N_k^2(t'_k-t_k))\|_{\dot{H}^1_x(\mathbb{R}^3)}\lesssim\varepsilon.
\end{equation*}
Since, by definition, $\phi'(z)=\overline{N}^{1/2}v(\overline{N}z-\overline{x},-\overline{t})$, this follows provided that
\begin{equation*}
\lim_{k\to\infty} N_k\Psi_I^{-1}({h'_k}^{-1}h_k\cdot \Psi_I(x/N_k))-x=\overline{x}\qquad\text{ for any }x\in \mathbb{R}^3.
\end{equation*}
This last claim follows by explicit computations using \eqref{mb3} and the definition \eqref{PsiCoord}.
\medskip

(ii) We analyze three cases:

{\bf{Case 1: $\mathcal{O},\mathcal{O}'\in \mathcal{F}_h$.}} We may assume that $\phi,\psi\in C^\infty_0(\mathbb{H}^3)$ and select a subsequence such that either
\begin{equation}\label{plu1}
\lim_{k\to\infty}|t_k-t'_k|=\infty
\end{equation}
or
\begin{equation}\label{plu5}
\lim_{k\to\infty} t_k-t'_k=\overline{t}\in\mathbb{R},\qquad \lim_{k\to\infty}d(h_k\cdot\mathbf{0},h'_k\cdot\mathbf{0})=\infty.
\end{equation}
Using \eqref{dispersive} it follows that
\begin{equation*}
\begin{split}
&\|\Pi_{t,h}\phi\|_{L^6(\mathbb{H}^3)}+\|\Pi_{t,h}(\Delta_\g\phi)\|_{L^6(\mathbb{H}^3)}\lesssim_{\phi}(1+|t|)^{-1}\\
&\|\Pi_{t,h}\psi\|_{L^6(\mathbb{H}^3)}+\|\Pi_{t,h}(\Delta_\g\psi)\|_{L^6(\mathbb{H}^3)}\lesssim_{\psi}(1+|t|)^{-1},
\end{split}
\end{equation*}
for any $t\in\mathbb{R}$ and $h\in\mathbb{G}$. Thus
\begin{equation}\label{plu4}
\|\widetilde{\phi}_{\mathcal{O}_k}\widetilde{\psi}_{\mathcal{O}'_k}\|_{L^3(\mathbb{H}^3)}\leq \|\Pi_{t_k,h_k}\phi\|_{L^6(\mathbb{H}^3)}\|\Pi_{t'_k,h'_k}\psi\|_{L^6(\mathbb{H}^3)}\lesssim _{\phi,\psi}(1+|t_k|)^{-1}(1+|t'_k|)^{-1},
\end{equation}
and
\begin{equation*}
\begin{split}
\Big|\int_{\mathbb{H}^3}\D^\al\widetilde{\phi}_{\mathcal{O}_k}&\overline{\D_\al\widetilde{\psi}_{\mathcal{O}'_k}}\,d\mu\Big|=\Big|\int_{\mathbb{H}^3}\Delta_\g\widetilde{\phi}_{\mathcal{O}_k}\cdot \overline{\widetilde{\psi}_{\mathcal{O}'_k}}\,d\mu\Big|=\Big|\int_{\mathbb{H}^3}\pi_{{h'_k}^{-1}h_k}e^{-i(t_k-t'_k)\Delta_\g}(\Delta_\g\phi)\cdot \overline{\psi}\,d\mu\Big|\\
&\lesssim\|\pi_{{h'_k}^{-1}h_k}e^{-i(t_k-t'_k)\Delta_\g}(\Delta_\g\phi)\|_{L^6(\mathbb{H}^3)}\|\psi\|_{L^{6/5}(\mathbb{H}^3)}\lesssim_{\phi,\psi}(1+|t_k-t'_k|)^{-1}.
\end{split}
\end{equation*}
The claim \eqref{mb50} follows if the selected subsequence verifies \eqref{plu1}.

If the selected subsequence verifies \eqref{plu5} then, as before,
\begin{equation*}
\begin{split}
\Big|&\int_{\mathbb{H}^3}\D^\al\widetilde{\phi}_{\mathcal{O}_k}\overline{\D_\al\widetilde{\psi}_{\mathcal{O}'_k}}\,d\mu\Big|=\Big|\int_{\mathbb{H}^3}\pi_{{h'_k}^{-1}h_k}e^{-i(t_k-t'_k)\Delta_g}\phi\cdot \overline{\Delta_\g\psi}\,d\mu\Big|\\
&\lesssim\|\Delta_\g\psi\|_{L^2(\mathbb{H}^3)}\cdot \|e^{-i\overline{t}\Delta_\g}\phi-e^{-i(t_k-t'_k)\Delta_\g}\phi\|_{L^2(\mathbb{H}^3)}+\int_{\mathbb{H}^3}|e^{-i\overline{t}\Delta_g}\phi|\cdot |\pi_{h_k^{-1}h'_k}\Delta_\g\psi|\,d\mu.
\end{split}
\end{equation*}
The first limit in \eqref{mb50} follows. Using the bound \eqref{plu4}, the second limit in \eqref{mb50} also follows, up to a subsequence, if $\limsup_{k\to\infty}|t_k|=\infty$. Otherwise, we may assume that $\lim_{k\to\infty}t_k=T$, $\lim_{k\to\infty}t'_k=T'=T-\overline{t}$ and estimate
\begin{equation*}
\begin{split}
\|\widetilde{\phi}_{\mathcal{O}_k}\widetilde{\psi}_{\mathcal{O}'_k}\|_{L^3(\mathbb{H}^3)}&=\|e^{-it_k\Delta_\g}\pi_{h_k}\phi\cdot e^{-it'_k\Delta_\g}\pi_{h'_k}\psi\|_{L^3(\mathbb{H}^3)}\\
\lesssim_{\phi,\psi}&\|e^{-it_k\Delta_\g}\phi-e^{-iT\Delta_\g}\phi\|_{L^6(\mathbb{H}^3)}+\|e^{-it'_k\Delta_\g}\psi-e^{-iT'\Delta_\g}\psi\|_{L^6(\mathbb{H}^3)}\\
&+\|e^{-iT\Delta_\g}\phi\cdot \pi_{h_k^{-1}h'_k}(e^{-iT'\Delta_\g}\psi)\|_{L^3(\mathbb{H}^3)}.
\end{split}
\end{equation*}
The second limit in \eqref{mb50} follows in this case as well.

{\bf{Case 2: $\mathcal{O}\in\mathcal{F}_h$, $\mathcal{O}'\in \mathcal{F}_e$.}} We may assume that $\phi\in C^\infty_0(\mathbb{H}^3)$ and $\psi\in C^\infty_0(\mathbb{R}^3)$. We estimate
\begin{equation*}
\begin{split}
\Big|&\int_{\mathbb{H}^3}\D^\al\widetilde{\phi}_{\mathcal{O}_k}\overline{\D_\al\widetilde{\psi}_{\mathcal{O}'_k}}\,d\mu\Big|=\Big|\int_{\mathbb{H}^3}\Pi_{t_k,h_k}(\Delta_g\phi)\cdot \overline{\Pi_{t'_k,h'_k}(T_{N'_k}\psi)}\,d\mu\Big|\lesssim_\phi\|T_{N'_k}\psi\|_{L^2(\mathbb{H}^3)}\lesssim_{\phi,\psi}{N'_k}^{-1}
\end{split}
\end{equation*}
and
\begin{equation*}
\begin{split}
\|\widetilde{\phi}_{\mathcal{O}_k}\widetilde{\psi}_{\mathcal{O}'_k}\|_{L^3(\mathbb{H}^3)}&\leq \|\Pi_{t_k,h_k}\phi\|_{L^\infty(\mathbb{H}^3)}\cdot\|\Pi_{t'_k,h'_k}(T_{N'_k}\psi)\|_{L^3(\mathbb{H}^3)}\\
&\lesssim \|\Delta_\g\phi\|_{L^2(\mathbb{H}^3)}\cdot\|(-\Delta_\g)^{1/4}(T_{N'_k}\psi)\|_{L^2(\mathbb{H}^3)}\\
&\lesssim_{\phi,\psi}{N'_k}^{-1/2}.
\end{split}
\end{equation*}
The limits in \eqref{mb50} follow.
\medskip

{\bf{Case 3: $\mathcal{O},\mathcal{O}'\in \mathcal{F}_e$.}} We may assume that $\phi,\psi\in C^\infty_0(\mathbb{R}^3)$ and select a subsequence such that either
\begin{equation}\label{plu10}
\lim_{k\to\infty}N_k/N'_k=0,
\end{equation}
or
\begin{equation}\label{plu11}
\lim_{k\to\infty}N_k/N'_k=\overline{N}\in(0,\infty),\qquad\lim_{k\to\infty} N_k^2|t_k-t'_k|=\infty,
\end{equation}
or
\begin{equation}\label{plu12}
\lim_{k\to\infty}N_k/N'_k=\overline{N}\in(0,\infty),\quad\lim_{k\to\infty} N_k^2(t_k-t'_k)=\overline{t}\in\mathbb{R},\quad \lim_{k\to\infty}N_kd(h_k\cdot\mathbf{0},h'_k\cdot\mathbf{0})=\infty.
\end{equation}

Assuming \eqref{plu10} we estimate, as in {\bf{Case 2}}, 
\begin{equation*}
\begin{split}
\Big|\int_{\mathbb{H}^3}\D^\al\widetilde{\phi}_{\mathcal{O}_k}\overline{\D_\al\widetilde{\psi}_{\mathcal{O}'_k}}\,d\mu\Big|&=\Big|\int_{\mathbb{H}^3}\Pi_{t_k,h_k}(\Delta_g(T_{N_k}\phi))\cdot \overline{\Pi_{t'_k,h'_k}(T_{N'_k}\psi)}\,d\mu\Big|\\
&\lesssim\|\Delta_\g(T_{N_k}\phi)\|_{L^2(\mathbb{H}^3)}\|T_{N'_k}\psi\|_{L^2(\mathbb{H}^3)}\\
&\lesssim_{\phi,\psi}N_k{N'_k}^{-1}
\end{split}
\end{equation*}
and
\begin{equation*}
\begin{split}
\|\widetilde{\phi}_{\mathcal{O}_k}\widetilde{\psi}_{\mathcal{O}'_k}\|_{L^3(\mathbb{H}^3)}&\leq \|\Pi_{t_k,h_k}(T_{N_k}\phi)\|_{L^9(\mathbb{H}^3)}\cdot\|\Pi_{t'_k,h'_k}(T_{N'_k}\psi)\|_{L^{9/2}(\mathbb{H}^3)}\\
&\lesssim \|(-\Delta_\g)^{7/12}(T_{N_k}\phi)\|_{L^2(\mathbb{H}^3)}\cdot\|(-\Delta_\g)^{5/12}(T_{N'_k}\psi)\|_{L^2(\mathbb{H}^3)}\\
&\lesssim_{\phi,\psi}N_k^{1/6}{N'_k}^{-1/6}.
\end{split}
\end{equation*}
The limits in \eqref{mb50} follow in this case.

To prove the limit \eqref{mb50} assuming \eqref{plu11}, we estimate first, using \eqref{dispersive},
\begin{equation}\label{plu15}
\|\Pi_{t,h}(T_N f)\|_{L^6(\mathbb{H}^3)}\lesssim_f (1+N^2|t|)^{-1},
\end{equation}
for any $t\in\mathbb{R}$, $h\in\mathbb{G}$, $N\in[0,\infty)$, and $f\in C^\infty_0(\mathbb{R}^3)$. Thus
\begin{equation*}
\begin{split}
\|\widetilde{\phi}_{\mathcal{O}_k}\widetilde{\psi}_{\mathcal{O}'_k}\|_{L^3(\mathbb{H}^3)}&\leq \|\Pi_{t_k,h_k}(T_{N_k}\phi)\|_{L^6(\mathbb{H}^3)}\|\Pi_{t'_k,h'_k}(T_{N'_k}\psi)\|_{L^6(\mathbb{H}^3)}\\
&\lesssim _{\phi,\psi}(1+N_k^2|t_k|)^{-1}(1+{N'_k}^2|t'_k|)^{-1},
\end{split}
\end{equation*}
and
\begin{equation*}
\begin{split}
\Big|\int_{\mathbb{H}^3}\D^\al\widetilde{\phi}_{\mathcal{O}_k}\overline{\D_\al\widetilde{\psi}_{\mathcal{O}'_k}}\,d\mu\Big|&=\Big|\int_{\mathbb{H}^3}\widetilde{\phi}_{\mathcal{O}_k}\cdot \overline{\Delta_\g\widetilde{\psi}_{\mathcal{O}'_k}}\,d\mu\Big|\\
&=\Big|\int_{\mathbb{H}^3}\pi_{{h'_k}^{-1}h_k}e^{-i(t_k-t'_k)\Delta_\g}(T_{N_k}\phi)\cdot \overline{\Delta_\g(T_{N'_k}\psi)}\,d\mu\Big|\\
&\lesssim\|\pi_{{h'_k}^{-1}h_k}e^{-i(t_k-t'_k)\Delta_\g}(T_{N_k}\phi)\|_{L^6(\mathbb{H}^3)}\|\Delta_\g(T_{N'_k}\psi)\|_{L^{6/5}(\mathbb{H}^3)}\\
&\lesssim_{\phi,\psi}(1+N_k^2|t_k-t'_k|)^{-1}.
\end{split}
\end{equation*}
The claim \eqref{mb50} follows if the selected subsequence verifies \eqref{plu11}.

Finally, it remains to prove the limit \eqref{mb50} if the selected subsequence verifies \eqref{plu12}. For this we will use the following claim: if $(g_k,M_k)_{k\geq 1}\in\mathbb{G}\times[1,\infty)$, $\lim_{k\to\infty}M_k=\infty$, $\lim_{k\to\infty}M_kd(g_k\cdot{\bf{0}},{\bf{0}})=\infty$, and $f,g\in\dot{H}^1(\mathbb{R}^3)$ then
\begin{equation}\label{plu30}
\lim_{k\to\infty}\Big|\int_{\mathbb{H}^3}\pi_{g_k}(-\Delta_\g)^{1/2}(T_{M_k}f)\cdot (-\Delta_\g)^{1/2}(T_{M_k}g)\,d\mu\Big|+\|\pi_{g_k}(T_{M_k}f)\cdot (T_{M_k}g)\|_{L^3(\mathbb{H}^3)}=0.
\end{equation}

Assuming this, we can complete the proof of \eqref{mb50}. It follows from \eqref{mb6} that if $f\in\dot{H}^1(\mathbb{R})$ and $\{s_k\}_{k\geq 1}$ is a sequence with the property that $\lim_{k\to\infty}N_k^2s_k=\overline{s}\in\mathbb{R}$ then
\begin{equation}\label{plu31}
\lim_{k\to\infty}\|e^{-is_k\Delta_\g}(T_{N_k}f)-T_{N'_k}f'\|_{H^1(\mathbb{H}^3)}=0,
\end{equation}
where $f'(x)=\overline{N}^{1/2}(e^{-i\overline{s}\Delta}f)(\overline{N}x)$. We estimate
\begin{equation*}
\begin{split}
\Big|\int_{\mathbb{H}^3}\D^\al\widetilde{\phi}_{\mathcal{O}_k}\overline{\D_\al\widetilde{\psi}_{\mathcal{O}'_k}}\,d\mu\Big|&=\Big|\int_{\mathbb{H}^3}(-\Delta_\g)^{1/2}\pi_{{h'_k}^{-1}h_k}e^{-i(t_k-t'_k)\Delta_\g}(T_{N_k}\phi)\cdot \overline{(-\Delta_\g)^{1/2}(T_{N'_k}\psi)}\,d\mu\Big|\\
&\lesssim\Big|\int_{\mathbb{H}^3}(-\Delta_\g)^{1/2}\pi_{{h'_k}^{-1}h_k}(T_{N'_k}\phi')\cdot \overline{(-\Delta_\g)^{1/2}(T_{N'_k}\psi)}\,d\mu\Big|\\
&+\|\psi\|_{\dot{H}^1(\mathbb{R}^3)}\cdot\|\pi_{{h'_k}^{-1}h_k}e^{-i(t_k-t'_k)\Delta_\g}(T_{N_k}\phi)-\pi_{{h'_k}^{-1}h_k}(T_{N'_k}\phi')\|_{H^1(\mathbb{H}^3)}.
\end{split}
\end{equation*}
In view of \eqref{plu30} and \eqref{plu31}, both terms in the expression above converge to $0$ as $k\to\infty$, as desired. If $\lim_{k\to\infty}N_k^2|t_k|=\infty$ then, using \eqref{plu15}, we estimate
\begin{equation*}
\|\widetilde{\phi}_{\mathcal{O}_k}\widetilde{\psi}_{\mathcal{O}'_k}\|_{L^3(\mathbb{H}^3)}\leq \|\Pi_{t_k,h_k}(T_{N_k}\phi)\|_{L^6(\mathbb{H}^3)}\|\Pi_{t'_k,h'_k}(T_{N'_k}\psi)\|_{L^6(\mathbb{H}^3)}\lesssim _{\phi,\psi}(1+N_k^2|t_k|)^{-1},
\end{equation*}
which converges to $0$ as $k\to\infty$. Otherwise, up to a subsequence, we may assume that $\lim_{k\to\infty}N_k^2t_k=T\in\mathbb{R}$, $\lim_{k\to\infty}$ and write
\begin{equation*}
\|\widetilde{\phi}_{\mathcal{O}_k}\widetilde{\psi}_{\mathcal{O}'_k}\|_{L^3(\mathbb{H}^3)}=\|\pi_{{h'_k}^{-1}h_k}e^{-it_k\Delta_g}(T_{N_k}\phi)\cdot e^{-it'_k\Delta_g}(T_{N'_k}\psi)\|_{L^3(\mathbb{H}^3)}.
\end{equation*}
This converges to $0$ as $k\to\infty$, using \eqref{plu30} and \eqref{plu31}, as desired.

It remains to prove the claim \eqref{plu30}. In view of the $\dot{H}^1(\mathbb{R}^3)\to H^1(\mathbb{H}^3)$ boundedness of the operators $T_N$, we may assume that $f,g\in C^\infty_0(\mathbb{R}^3)$. In this case, the supports of the functions $\pi_{g_k}(T_{M_k}f)$ and $T_{M_k}g$ become disjoint for $k$ sufficiently large (due to the assumption $\lim_{k\to\infty}M_kd(g_k\cdot{\bf{0}},{\bf{0}})=\infty$), and the limit \eqref{plu30} follows.

(iii) By boundedness of $T_{N_k}$, it suffices to consider the case when $\phi,\psi\in C^\infty_0(\mathbb{R}^3)$. In this case, we have that
$$\Vert \nabla_\g\left(T_{N_k}\phi-N_k^{1/2}\phi(N_k\Psi_I^{-1}\cdot)\right)\Vert_{L^2(\mathbb{H}^3)}\to 0$$
as $k\to\infty$. Hence, by unitarity of $\Pi_{t_k,h_k}$, it suffices to compute
\begin{equation*}
\begin{split}
&\lim_{k\to\infty}N_k\langle\nabla_\g\left(\phi(N_k\Psi^{-1}\cdot)\right),\nabla_\g\left(\psi(N_k\Psi_I^{-1}\cdot)\right)\rangle_{L^2\times L^2(\mathbb{H}^3)}=\int_{\mathbb{R}^3}\nabla\phi(x)\cdot\nabla\overline{\psi}(x)dx,
\end{split}
\end{equation*}
which follows after a change of variables and use of the dominated convergence theorem.
\end{proof}

Our main result in this section is the following.

\begin{proposition}\label{ProfileDec}
Assume that $(f_k)_{k\geq 1}$ is a bounded sequence in $H^1(\mathbb{H}^3)$. Then there are sequences of pairs $(\phi^\mu,\mathcal{O}^\mu)\in\dot{H}^1(\mathbb{R}^3)\times\mathcal{F}_e$ and $(\psi^\nu,\widetilde{\mathcal{O}}^\nu)\in H^1(\mathbb{H}^3)\times\mathcal{F}_h$, $\mu,\nu=1,2,\ldots$, such that, up to a subsequence, for any $J\geq 1$,
\begin{equation}\label{LinProfileDec}
f_k=\sum_{1\le\mu\le J}\widetilde{\phi}^\mu_{\mathcal{O}^\mu_k}+\sum_{1\le\nu\le J}\widetilde{\psi}^\nu_{\widetilde{\mathcal{O}}^\nu_k}+r_k^J,
\end{equation}
where $\widetilde{\phi}^\mu_{\mathcal{O}^\mu_k}$ and $\widetilde{\psi}^\nu_{\widetilde{\mathcal{O}}^\nu}$ are the associated profiles in Definition \ref{DefPro}, and
\begin{equation}\label{rkJSmall}
\lim_{J\to\infty}\limsup_{k\to\infty}\left(\sup_{N\ge 1,t\in\mathbb{R},x\in\mathbb{H}^3}N^{-1/2}\vert P_Ne^{it\Delta_\g}r_k^J\vert(x)\right)=0.
\end{equation}
Moreover the frames $\{\mathcal{O}^\mu\}_{\mu\geq 1}$ and $\{\widetilde{\mathcal{O}}^\nu\}_{\nu\geq 1}$ are  pairwise orthogonal. Finally, the decomposition is asymptotically orthogonal in the sense that
\begin{equation}\label{AlmostOrtho}
\begin{split}
\lim_{J\to\infty}\limsup_{k\to\infty}\Big|E^1(f_k)-\sum_{1\leq\mu\leq J}E^1(\widetilde{\phi}^\mu_{\mathcal{O}^\mu_k})-\sum_{1\leq\nu\leq J}E^1(\widetilde{\psi}^\nu_{\widetilde{\mathcal{O}}^\nu_k})-E^1(r_k^J)\Big|= 0,
\end{split}
\end{equation}
where $E^1$ is the energy defined in \eqref{conserve}.
\end{proposition}

The profile decomposition in Proposition \ref{ProfileDec} is a consequence of the following finitary decomposition.

\begin{lemma}\label{FiniteProfileDec}
Let $(f_k)_{k\geq 1}$ be a bounded sequence of functions in $H^1(\mathbb{H}^3)$ and let $\delta\in(0,\delta_0]$ be sufficiently small. Up to passing to a subsequence, the sequence $(f_k)_{k\geq 1}$ can be decomposed into $2J+1=O(\delta^{-2})$ terms
\begin{equation}\label{DecOfF}
f_k=\sum_{1\le \mu\le J} \widetilde{\phi}^\mu_{\mathcal{O}^\mu_k}+\sum_{1\le \nu\le J} \widetilde{\psi}^\nu_{\widetilde{\mathcal{O}}^\nu_k}+r_k,
\end{equation}
where $\widetilde{\phi}^\mu_{\mathcal{O}^\mu_k}$ (respectively $\widetilde{\psi}^\nu_{\widetilde{\mathcal{O}}^\nu_k}$) are Euclidean (respectively hyperbolic) profiles associated to sequences $(\phi^\mu,\mathcal{O}^\mu)\in\dot{H}^1(\mathbb{R}^3)\times\mathcal{F}_e$ (respectively $(\psi^\nu,\widetilde{\mathcal{O}}^\nu)\in H^1(\mathbb{H}^3)\times\mathcal{F}_h$) as  in Definition \ref{DefPro}.

Moreover the remainder $r_k$ is absent from all the frames $\mathcal{O}^\mu$, $\widetilde{\mathcal{O}}^\nu$, $1\le\mu,\nu\le J$ and
\begin{equation}\label{r_kSmall}
\limsup_{k\to\infty}\left(\sup_{N\ge 1,t\in\mathbb{R},x\in\mathbb{H}^3}N^{-1/2}\vert e^{it\Delta_\g}P_Nr_k\vert(x)\right)\le \delta.
\end{equation}
In addition, the frames $\mathcal{O}^\mu$ and $\widetilde{\mathcal{O}}^\nu$ are pairwise orthogonal, and the decomposition is asymptotically orthogonal in the sense that
\begin{equation}\label{orthdecomp1}
\begin{split}
\Vert\nabla_\g f_k\Vert_{L^2}^2=\sum_{1\le \mu\le J}\Vert \nabla_\g  \widetilde{\phi}^\mu_{\mathcal{O}^\mu_k}\Vert_{L^2}^2+\sum_{1\le\nu\le J}\Vert\nabla_\g \widetilde{\psi}^\nu_{\widetilde{\mathcal{O}}^\nu_k}\Vert_{L^2}^2+\Vert \nabla_\g r_k\Vert_{L^2}^2+o_k(1)
\end{split}
\end{equation}
where $o_k(1)\to0$ as $k\to\infty$.
\end{lemma}

We show first how to prove Proposition \ref{ProfileDec} assuming the finitary decomposition of Lemma \ref{FiniteProfileDec}.

\begin{proof}[Proof of Proposition \ref{ProfileDec}]
We apply Lemma \ref{FiniteProfileDec} repeatedly for $\delta=2^{-l}$, $l=1,2,\dots$ and we obtain the result except for \eqref{AlmostOrtho}. To prove this, it suffices from \eqref{orthdecomp1} to prove the addition of the $L^6$-norms. But from Lemma \ref{lem12} and \eqref{rkJSmall}, we see that
\begin{equation*}
\begin{split}
\limsup_{J\to\infty}\limsup_{k\to\infty}\Vert r^J_k\Vert_{L^6(\mathbb{H}^3)}=0
\end{split}
\end{equation*}
so that
\begin{equation}\label{alex1}
\limsup_{J\to\infty}\limsup_{k\to\infty} \left(\left\vert\Vert f_k\Vert_{L^6}^6- \Vert f_k-r^J_k\Vert_{L^6}^6\right\vert+\Vert r^J_k\Vert_{L^6}^6\right)=0.
\end{equation}
Now, for fixed $J$, we see that
\begin{equation*}
\begin{split}
\left\vert \vert f_k-r^J_k\vert^6-\sum_{1\le\mu\le J}\vert \widetilde{\phi}^\mu_{\mathcal{O}^\mu_k}\vert^6-\sum_{1\le\nu\le J}\vert \widetilde{\psi}^\nu_{\widetilde{\mathcal{O}}^\nu_k}\vert^6\right\vert&\lesssim_J\sum_{1\le\alpha\neq\beta\le J}\vert \widetilde{\phi}^\alpha_{\mathcal{O}^\alpha_k}\vert\vert \widetilde{\phi}^\beta_{\mathcal{O}^\beta_k}\vert^5+\sum_{1\le\alpha\neq\beta\le J}\vert \widetilde{\psi}^\alpha_{\widetilde{\mathcal{O}}^\alpha_k}\vert\vert \widetilde{\psi}^\beta_{\widetilde{\mathcal{O}}^\beta_k}\vert^5\\
&+\sum_{1\le\mu,\nu\le J}\left(\vert \widetilde{\phi}^\mu_{\mathcal{O}^\mu_k}\vert\vert \widetilde{\psi}^\nu_{\widetilde{\mathcal{O}}^\nu_k}\vert^5+\vert\widetilde{\phi}^\mu_{\mathcal{O}^\mu_k}\vert^5\vert \widetilde{\psi}^\nu_{\widetilde{\mathcal{O}}^\nu_k}\vert\right)
\end{split}
\end{equation*}
so that
\begin{equation*}
\begin{split}
\left\vert \Vert f_k-r^J_k\Vert_{L^6}^6-\sum_{1\le\mu\le J}\Vert \widetilde{\phi}^\mu_{\mathcal{O}^\mu_k}\Vert_{L^6}^6-\sum_{1\le\nu\le J}\Vert \widetilde{\psi}^\nu_{\widetilde{\mathcal{O}}^\nu_k}\Vert_{L^6}^6\right\vert
&\lesssim_{J}\sum_{\alpha,\beta}\Vert f^\alpha_kf^\beta_k\Vert_{L^3}
\end{split}
\end{equation*}
where the summation ranges over all pairs $(f^\alpha_k,f^\beta_k)$ of profiles such that $f^\alpha_k\ne f^\beta_k$ and where we have used the fact that the $L^6$ norm of each profile is bounded uniformly. From Lemma \ref{equiv} (ii), we see that this converges to $0$ as $k\to\infty$. The identity \eqref{AlmostOrtho} follows using also \eqref{alex1}.
\end{proof}

\subsection{Proof of Lemma \ref{FiniteProfileDec}}
\label{TechProof}

For $(g_k)_k$ a bounded sequence in $H^1(\mathbb{H}^3)$, we let
\begin{equation}\label{DefOfDelta}
\delta((g_k)_k)=\sup_{\widetilde{\mathcal{O}}_k}\limsup_{k\to\infty}N_k^{-\frac{1}{2}}\left\vert P_{N_k}\left(e^{it_k\Delta_\g}g_k\right)\right\vert(h_k\cdot\bf{0})
\end{equation}
where the supremum is taken over all sequences $\widetilde{\mathcal{O}}_k=(N_k,t_k,h_k)_k$ with $N_k\ge 1$, $t_k\in\mathbb{R}$ and $h_k\in\mathbb{G}$. If $\delta((f_k)_k)\le\delta$, then we let $J=0$ and $f_k=r_k$ and Lemma \ref{FiniteProfileDec} follows. Otherwise, we use inductively the following:

\medskip

{\bf{Claim:}} Assume $(g_k)_k$ is a bounded sequence in $H^1(\mathbb{H}^3)$ which is absent from a family of frames $(\mathcal{O}^\alpha)_{\alpha\le A}$ and such that $\delta((g_k)_k)\ge\delta$. Then there exists a new frame $\mathcal{O}^\prime$ which is orthogonal to $\mathcal{O}^\alpha$ for all $\alpha\le A$ and a profile $\widetilde{\phi}_{\mathcal{O}^\prime_k}$ of free energy
\begin{equation}\label{NontrivialWL}
\limsup_{k\to\infty}\Vert\nabla_\g \widetilde{\phi}_{\mathcal{O}^\prime_k}\Vert_{L^2}\gtrsim\delta
\end{equation}
such that, after passing to a subsequence, $g_k-\widetilde{\phi}_{\mathcal{O}^\prime_k}$ is absent from the frames $\mathcal{O}^\prime$ and $\mathcal{O}^\alpha$, $\alpha\le A$.

\medskip

Once we have proved the claim, Lemma \ref{FiniteProfileDec} follows by applying repeatedly the above procedure. Indeed, we let $(f^\alpha_k)_k$ be defined as follows: $(f^0_k)_k=(f_k)_k$ and if $\delta((f^\alpha_k)_k)\ge\delta$, then apply the above claim to $(f^\alpha_k)_k$ to get a new sequence $$f^{\alpha+1}_k=f^{\alpha}_k-\widetilde{\phi}_{\mathcal{O}^{\alpha+1}_k}.$$ By induction, $(f^\alpha_k)_k$ is absent from all the frames $\mathcal{O}^\beta$, $\beta\le\alpha$. This procedure stops after a finite number ($O(\delta^{-2})$) of steps. Indeed since
$f^{\alpha}_k=f^{\alpha-1}_k-\widetilde{\phi}_{\mathcal{O}^\alpha_k}$ is absent from $\mathcal{O}^\alpha_k$, we get from \eqref{WeakConvTo0} that
\begin{equation*}
\begin{split}
\Vert\nabla_\g f^{\alpha-1}_k\Vert_{L^2}^2&=\Vert\nabla_\g f^{\alpha}_k\Vert_{L^2}^2+\Vert\nabla_\g \widetilde{\phi}_{\mathcal{O}^\alpha_k}\Vert_{L^2}^2+2\langle f^{\alpha}_k,\widetilde{\phi}_{\mathcal{O}^\alpha_k}\rangle_{H^1\times H^1(\mathbb{H}^3)}\\
&=\Vert\nabla_\g f^{\alpha}_k\Vert_{L^2}^2+\Vert\nabla_\g \widetilde{\phi}_{\mathcal{O}^\alpha_k}\Vert_{L^2}^2+o_k(1)
\end{split}
\end{equation*}
and therefore by induction,
\begin{equation*}
\Vert \nabla_\g f_k\Vert_{L^2}^2=\sum_{1\le\alpha\le A}\Vert\nabla_\g \widetilde{\phi}_{\mathcal{O}^\alpha}\Vert_{L^2}^2+\Vert \nabla_\g f_k^A\Vert_{L^2}^2 +o_k(1).
\end{equation*}
Since each profile has a free energy $\gtrsim\delta$, this is a finite process and Lemma \ref{FiniteProfileDec} follows.

\medskip

Now we prove the claim. By hypothesis, there exists a sequence $\widetilde{\mathcal{O}}_k=(N_k,t_k,h_k)_k$ such that the $\limsup_{k\to\infty}$ in \eqref{DefOfDelta} is greater than $\delta/2$. If $\limsup_{k\to\infty}N_k=\infty$, then, up to passing to a subsequence, we may assume that $\{\widetilde{\mathcal{O}}_k\}_{k\ge 1}=\mathcal{O}^\prime$ is a Euclidean frame.
Otherwise, up to passing to a subsequence, we may assume that $N_k\to N\ge 1$ and we let $\mathcal{O}^\prime=\{(1,t_k,h_k)_k\}_{k\ge 1}$ be a hyperbolic frame.
In all cases, we get a
frame $\mathcal{O}^\prime=\{(M_k,t_k,h_k)_k\}_{k\ge 1}$ such that
\begin{equation}\label{NonzeroScalarProduct}
\begin{split}
\delta/2&\le\limsup_{k\to\infty}N_k^{-\frac{1}{2}}\left\vert P_{N_k}\left(e^{it_k\Delta_\g}\right)g_k\right\vert(h_k\cdot\bf{0})\\
&=\limsup_{k\to\infty}\langle \Pi_{-t_k,h_k^{-1}}g_k,N_k^{-\frac{1}{2}} P_{N_k}(\delta_{\bf{0}})\rangle_{L^2\times L^2(\mathbb{H}^3)}
\end{split}
\end{equation}
for some sequence $N_k$ comparable to $M_k$.

Now, we claim that there exists a profile $\widetilde{f}_{\mathcal{O}_k^\prime}$ associated to the frame $\mathcal{O}^\prime$ such that
\begin{equation*}
\limsup_{k\to\infty}
\Vert \nabla_\g \widetilde{f}_{\mathcal{O}_k^\prime}\Vert_{L^2}\lesssim 1
\end{equation*}
and
\begin{equation*}
\Pi_{-t_k,h_k^{-1}}\widetilde{f}_{\mathcal{O}_k^\prime}-N_k^{-\frac{5}{2}}e^{-N_k^{-2}\Delta_\g}(\delta_{\bf{0}})\to0
\end{equation*}
strongly in $H^1(\mathbb{H}^3)$. Indeed, if $\mathcal{O}^\prime$ is a hyperbolic frame, then
$f:=N^{-\frac{5}{2}} e^{-N^{-2}\Delta_\g}\delta_{\bf{0}}$. If $N_k\to\infty$, we let $f(x):=(4\pi)^{-\frac{3}{2}}e^{-\vert x\vert^2/4}=e^{-\Delta}\delta_0$. By unitarity of $\Pi$ it suffices to see that
\begin{equation}\label{ConvHK}
\Vert N_k^{-\frac{5}{2}} e^{N_k^{-2}\Delta_\g}\delta_{\bf{0}}-T_{N_k}f\Vert_{H^1(\mathbb{H}^3)}\to 0
\end{equation}
which follows by inspection of the explicit formula
\begin{equation*}
\left(e^{-z\Delta_\g}\delta_{\bf{0}}\right)(P)=\frac{1}{(4\pi z)^\frac{3}{2}}e^{-z}\frac{r}{\sinh r}e^{-\frac{r^2}{4z}}
\end{equation*}
for $r=d_\g({\bf{0}},P)$.

\medskip

Since $g_k$ is absent from the frames $\mathcal{O}^\alpha$, $\alpha\le A$, and we have a nonzero scalar product in \eqref{NonzeroScalarProduct}, we see from the discussion after Definition \ref{DefAbsent} that $\mathcal{O}^\prime$ is orthogonal to these frames.

\medskip

Now, in the case $\mathcal{O}^\prime$ is a hyperbolic frame, we let $\psi\in H^1(\mathbb{H}^3)$ be any weak limit of $\Pi_{-t_k,h_k^{-1}}g_k$. Then, passing to a subsequence, we may assume that for any $\varphi\in H^1(\mathbb{H}^3)$,
\begin{equation*}
\langle \nabla_\g\left(\Pi_{-t_k,h_k^{-1}}g_k-\psi\right),\nabla_\g\varphi\rangle_{L^2\times L^2}=\langle \nabla_\g\left(g_k-\Pi_{t_k,h_k}\psi\right),\nabla_\g\Pi_{t_k,h_k}\varphi\rangle_{L^2\times L^2}\to 0
\end{equation*}
so that $g^\prime_k=g_k-\Pi_{t_k,h_k}\psi$ is absent from $\mathcal{O}^\prime$. In particular,
we see from \eqref{NonzeroScalarProduct} that
\begin{equation*}
\begin{split}
\delta/2&\le \limsup_{k\to\infty}\langle \Pi_{-t_k,h_k^{-1}}g_k,\Delta_\g N^{-\frac{5}{2}} \left(e^{N^{-2}\Delta_\g}\delta_{\bf{0}}\right)\rangle_{L^2\times L^2}\\
&\le\langle\psi,\Delta_\g N^{-\frac{5}{2}} \left(e^{N^{-2}\Delta_\g}\delta_{\bf{0}}\right)\rangle_{L^2\times L^2}\lesssim \Vert\nabla_\g\psi\Vert_{L^2(\mathbb{H}^3)}
\end{split}
\end{equation*}
so that \eqref{NontrivialWL} holds. Finally, to prove that $g^\prime_k$ is also absent from the frames $\mathcal{O}^\alpha$, $1\le\alpha\le A$ it suffices by hypothesis to prove this for $\widetilde{\psi}_{\mathcal{O}^\prime_k}$, but this follows from Lemma \ref{equiv} (ii). 

\medskip

In the case $N_k\to\infty$, we first choose $R>0$ and we define
\begin{equation}\label{PhiR}
\phi^R_k(v)=\eta(v/R)N_k^{-\frac{1}{2}}\left(\Pi_{-t_k,h_k^{-1}}g_k\right)(\Psi_I(v/N_k)),
\end{equation}
where $\eta$ is a smooth cut-off function as in \eqref{rescaled}. This sequence satisfies
\begin{equation*}
\limsup_{k\to\infty}\Vert\nabla\phi^R_k\Vert_{L^2(\mathbb{R}^3)}\lesssim \limsup_{k\to\infty}\Vert\nabla_\g g_k\Vert_{L^2(\mathbb{H}^3)}
\end{equation*}
and therefore has a subsequence which is bounded in $\dot{H}^1(\mathbb{R}^3)$ uniformly in $R>0$. Passing to a subsequence, we can find a weak limit $\phi^R\in \dot{H}^1(\mathbb{R}^3)$. Since the bound is uniform in $R>0$, we can let $R\to\infty$ and find a weak limit $\phi$ such that
\begin{equation*}
\phi^R\rightharpoonup \phi
\end{equation*}
in $H^1_{loc}$ and $\phi \in\dot{H}^1(\mathbb{R}^3)$. Now, for $\varphi\in C^\infty_c(\mathbb{R}^3)$, we have that
\begin{equation*}
\Vert T_{N_k}\varphi-N_k^\frac{1}{2}\varphi(N_k\Psi_I^{-1})\Vert_{H^1(\mathbb{H}^3)}\to 0
\end{equation*}
as $k\to\infty$
and with Lemma \ref{equiv} (iii), we compute that
\begin{equation}\label{LastClaimProfileDec}
\begin{split}
\langle g_k,\Delta_\g\widetilde{\varphi}_{\mathcal{O}^\prime_k}\rangle_{L^2\times L^2(\mathbb{H}^3)}&=\langle \Pi_{-t_k,h_k^{-1}}g_k,\Delta_\g T_{N_k}\varphi\rangle_{L^2\times L^2(\mathbb{H}^3)}\\
&=\langle \Pi_{-t_k,h_k^{-1}}g_k,\Delta_\g N_k^\frac{1}{2}\varphi(N_k\Psi_I^{-1}\cdot)\rangle_{L^2\times L^2(\mathbb{H}^3)}+o_k(1)\\
&=\langle\phi,\Delta\varphi\rangle_{L^2\times L^2(\mathbb{R}^3)}+o_k(1)\\
&=-\langle\widetilde{\phi}_{\mathcal{O}^\prime_k},\widetilde{\varphi}_{\mathcal{O}^\prime_k}\rangle_{H^1\times H^1(\mathbb{H}^3)}+o_k(1).
\end{split}
\end{equation}
In particular, $g^\prime_k=g_k-\widetilde{\phi}_{\mathcal{O}^\prime_k}$ is absent from $\mathcal{O}^\prime$ and from \eqref{NonzeroScalarProduct}, we see that \eqref{NontrivialWL} holds. Finally, from Lemma \ref{equiv} (ii) again, $g^\prime_k$ is absent from all the previous frames.

This finishes the proof of the claim and hence the proof of the finitary statement.

\section{Proof of Proposition \ref{lem4}}\label{proofnew}

In this section, we first give the proof of Proposition \ref{lem4} assuming a few lemmas that we prove at the end.

\subsection{Proof of Proposition \ref{lem4}}

Using the time translation symmetry, we may assume that $t_k=0$ for all $k\ge 1$. We apply Proposition \ref{ProfileDec} to the sequence $(u_k(0))_k$ which is bounded in $H^1(\mathbb{H}^3)$ and we get sequences of pairs $(\phi^\mu,\mathcal{O}^\mu)\in\dot{H}^1(\mathbb{R}^3)\times\mathcal{F}_e$ and $(\psi^\nu,\widetilde{\mathcal{O}}^\nu)\in H^1(\mathbb{H}^3)\times\mathcal{F}_h$, $\mu,\nu=1,2,\ldots$, such that the conclusion of Proposition \ref{ProfileDec} holds. Up to using Lemma \ref{equiv} (i), we may assume that for all $\mu$, either $t^\mu_k=0$ for all $k$ or $ (N_k^\mu)^2\vert t^\mu_k\vert\to\infty$ and similarly, for all $\nu$, either $t^\nu_k=0$ for all $k$ or $\vert t^\nu_k\vert\to\infty$.

\medskip

{\bf Case I}: all profiles are trivial, $\phi^\mu=0$, $\psi^\nu=0$ for all $\mu,\nu$. In this case, we get from Strichartz estimates, \eqref{rkJSmall} and Lemma \ref{lem12} (ii)
that $u_k(0)=r^J_k$ satisfies
\begin{equation*}
\begin{split}
\Vert e^{it\Delta_\g}(u_k(0))\Vert_{Z(\mathbb{R})}
&\lesssim \Vert e^{it\Delta_\g}(u_k(0))\Vert_{L^6_tL^{18}_x}^\frac{3}{5}\Vert e^{it\Delta_\g}(u_k(0))\Vert_{L^\infty_tL^6_x}^\frac{2}{5}\\
&\lesssim \Vert \nabla u_k(0)\Vert_{L^2}^\frac{11}{15}\left(\sup_{N\ge 1,t,x}N^{-\frac{1}{2}}\vert e^{it\Delta_\g}P_N(u_k(0))\vert(x)\right)^\frac{4}{15}\to 0
\end{split}
\end{equation*}
as $k\to\infty$. Appling Lemma \ref{SmallDataScattering}, we see that
\begin{equation*}
\Vert u_k\Vert_{Z(\mathbb{R})}\le\Vert e^{it\Delta_\g}u_k(0)\Vert_{L^{10}_{t,x}(\mathbb{H}^3\times\mathbb{R})}+\Vert u_k-e^{it\Delta_\g}u_k(0)\Vert_{S^1(\mathbb{R})}\to0
\end{equation*}
as $k\to\infty$ which contradicts \eqref{CondForCompactnessBigSNorm}.

\medskip

Now, for every linear profile $\widetilde{\phi}^\mu_{\mathcal{O}^\mu_k}$ (resp. $\psi^\nu_{\widetilde{\mathcal{O}}^\nu_k}$), define the associated nonlinear profile $U_{e,k}^\mu$ (resp. $U_{h,k}^\nu$) as the maximal solution of \eqref{eq1} with initial data $U^\mu_{e,k}(0)=\widetilde{\phi}^\mu_{\mathcal{O}^\mu_k}$ (resp. $U_{h,k}^\nu(0)=\psi^\nu_{\widetilde{\mathcal{O}}^\nu_k}$). We may write $U^\gamma_k$ if we do not want to discriminate between Euclidean and hyperbolic profiles.

\medskip

We can give a more precise description of each nonlinear profile.
\begin{enumerate}
\item If $\mathcal{O}^\mu\in\mathcal{F}_e$ is a Euclidean frame, this is given in Lemma \ref{GEForEP}.
\item If $t^\nu_k=0$, letting $(I^\nu,W^\nu)$ be the maximal solution of \eqref{eq1} with initial data $W^\nu(0)=\psi^{\nu}$, we see that for any interval $J\subset\subset I^\nu$,
\begin{equation}\label{DescHypProfile1}
\Vert U^\nu_{h,k}(t)-\pi_{h_k^\nu}W^\nu(t-t_k^\nu)\Vert_{S^1(J)}\to 0
\end{equation}
as $k\to\infty$ (indeed, this is identically $0$ in this case).
\item If $t^\nu_k\to+\infty$, then we define $(I^\nu,W^\nu)$ to be the maximal solution of \eqref{eq1} satisfying\footnote{Note that $(I^\nu,W^\nu)$ exists by Strichartz estimates and Lemma \ref{SmallDataScattering}.}
\begin{equation*}
\Vert W^\nu(t)-e^{it\Delta_\g}\psi^\nu\Vert_{H^1(\mathbb{H}^3)}\to 0
\end{equation*}
as $t\to-\infty$. Then, applying Proposition \ref{stability}, we see that on any interval $J=(-\infty,T)\subset\subset I^\nu$, we have \eqref{DescHypProfile1}. Using the time reversal symmetry $u(t,x)\to \overline{u}(-t,x)$, we obtain a similar description when $t^\nu_k\to-\infty$.
\end{enumerate}

\medskip

{\bf Case IIa}: there is only one Euclidean profile, i.e. there exists $\mu$ such that $u_k(0)=\widetilde{\phi}^\mu_{\mathcal{O}^\mu_k}+o_k(1)$ in $H^1(\mathbb{H}^3)$. Applying Lemma \ref{GEForEP}, we see that
$U^\mu_{e,k}$ is global with uniformly bounded $S^1$-norm for $k$ large enough. Then, using the stability Proposition \ref{stability} with $\tilde{u}=U^\mu_{e,k}$, we see that for all $k$ large enough,
\begin{equation*}
\Vert u_k\Vert_{Z(I)}\lesssim_{E_{max}}1
\end{equation*}
which contradicts \eqref{CondForCompactnessBigSNorm}.

\medskip

{\bf Case IIb}: there is only one hyperbolic profile, i.e. there is $\nu$ such that $u_k(0)=\widetilde{\psi}^\nu_{\widetilde{\mathcal{O}}^\nu_k}+o_k(1)$ in $H^1(\mathbb{H}^3)$. If we have that $t^\nu_k\to+\infty$, then, using Strichartz estimates, we see that
\begin{equation*}
\Vert \nabla_\g e^{it\Delta_\g}\Pi_{t^\nu_k,h_k^\nu}\psi^\nu\Vert_{L^{10}_tL^\frac{30}{13}_x(\mathbb{H}^3\times(-\infty,0))}=\Vert\nabla_\g e^{it\Delta_\g}\psi^\nu\Vert_{L^{10}_tL^\frac{30}{13}_x(\mathbb{H}^3\times(-\infty,-t^\nu_k))}\to0
\end{equation*}
as $k\to\infty$, which implies that $\Vert e^{it\Delta_\g}u_k(0)\Vert_{Z(-\infty,0)}\to 0$ as $k\to\infty$. Using again Lemma \ref{SmallDataScattering}, we see that, for $k$ large enough, $u_k$ is defined on $(-\infty,0)$ and $\Vert u_k\Vert_{Z(-\infty,0)}\to 0$ as $k\to\infty$, which contradicts \eqref{CondForCompactnessBigSNorm}. Similarly, $t^\nu_k\to-\infty$ yields a contradiction. Finally, if $t^\nu_k=0$, we get that
\begin{equation*}
\pi_{(h_k^\nu)^{-1}}u_k(0)\to \psi^\nu
\end{equation*}
converges strongly in $H^1(\mathbb{H}^3)$, which is the desired conclusion of the proposition.

\medskip

{\bf Case III}: there exists $\mu$ or $\nu$ and $\eta>0$ such that
\begin{equation}\label{OneNonTrivialNonFullProfile}
2\eta< \limsup_{k\to\infty}E^1(\widetilde{\phi}^\mu_{\mathcal{O}^\mu_k}), \limsup_{k\to\infty}E^1(\widetilde{\psi}^\nu_{\widetilde{\mathcal{O}}^\nu_k})<E_{max}-2\eta.
\end{equation}
Taking $k$ sufficiently large and maybe replacing $\eta$ by $\eta/2$, we may assume that \eqref{OneNonTrivialNonFullProfile} holds for all $k$.
In this case, we claim that for $J$ sufficiently large,
\begin{equation*}
\begin{split}
U^{app}_k&=\sum_{1\le\mu\le J}U^\mu_{e,k}+\sum_{1\le\nu\le J}U^\nu_{h,k}+e^{it\Delta_\g}r_k^J\\
&=U^J_{prof,k}+e^{it\Delta_\g}r_k^J
\end{split}
\end{equation*}
is a global approximate solution with bounded $Z$ norm for all $k$ sufficiently large.

\medskip

First, by Lemma \ref{GEForEP}, all the Euclidean profiles are global. Using \eqref{AlmostOrtho}, we see that for all $\nu$ and all $k$ sufficiently large, $E^1(U^{\nu}_{h,k})<E_{max}-\eta$. By \eqref{DescHypProfile1}, this implies that $E^1(W^\nu)<E_{max}-\eta$ so that by the definition of $E_{max}$, $W^\nu$ is global and by Proposition \ref{stability}, $U^\nu_{h,k}$ is global for $k$ large enough and
\begin{equation}\label{GDForHypP}
\Vert U^\nu_{h,k}(t)-\pi_{h_k}W^\nu(t-t_k^\nu)\Vert_{S^1(\mathbb{R})}\to 0
\end{equation}
as $k\to\infty$.
\medskip

Now we claim that
\begin{equation}\label{BoundedENorm}
\limsup_{k\to\infty}
\Vert\nabla_\g U^{app}_k\Vert_{L^\infty_tL^2_x}\le 4E_{max}^\frac{1}{2}
\end{equation}
is bounded uniformly in $J$.
Indeed, we first observe using \eqref{AlmostOrtho} that
\begin{equation*}
\begin{split}
\Vert \nabla_\g U^{app}_k\Vert_{L^\infty_tL^2_x}&\le \Vert \nabla_\g U^J_{prof,k}\Vert_{L^\infty_tL^2_x}+\Vert \nabla_\g r^J_k\Vert_{L^2_x}\\
&\le \Vert \nabla_\g U^J_{prof,k}\Vert_{L^\infty_tL^2_x}+\left(2E_{max}\right)^\frac{1}{2}.
\end{split}
\end{equation*}
Using Lemma \ref{Assump}, we get that for fixed $t$ and $J$,
\begin{equation*}
\begin{split}
\Vert \nabla_\g U^J_{prof,k}(t)\Vert_{L^2_x}^2
&\le  \sum_{1\le\gamma\le 2J}
\Vert \nabla_\g U^\gamma_k\Vert_{L^\infty_tL^2_x}^2+2\sum_{\gamma\ne\gamma^\prime}\langle\nabla_\g U^\gamma_k(t),\nabla_\g U^{\gamma^\prime}_k(t)\rangle_{L^2\times L^2}\\
&\le  2\sum_{1\le\gamma\le 2J}E^1(U^\gamma_k)+o_k(1)\le 2E_{max}+o_k(1),
\end{split}
\end{equation*}
where $o_k(1)\to0$ as $k\to\infty$ for fixed $J$.

\medskip

We also have that
\begin{equation}\label{UHasUniformlyBoundedSNorm}
\limsup_{k\to\infty} 
\Vert \nabla_\g U^{app}_k\Vert_{L^{10}_tL^\frac{30}{13}_x}\lesssim_{E_{max},\eta}1
\end{equation}
is bounded uniformly in $J$. Indeed, from \eqref{OneNonTrivialNonFullProfile} and \eqref{AlmostOrtho}, we see that for all $\gamma$ and all $k$ sufficiently large (depending maybe on $J$), $E^1(U^\gamma_k)<E_{max}-\eta$ and from the definition of $E_{max}$, we conclude that
$$\sup_\gamma\Vert U^\gamma_k\Vert_{Z(\mathbb{R})}\lesssim_{E_{max},\eta} 1.$$
Using Proposition \ref{stability}, we see that this implies that
\begin{equation*}
\sup_{\gamma}
\Vert \nabla_\g U^\gamma_k\Vert_{L^\frac{10}{3}_{t,x}}\lesssim_{E_{max},\eta} 1.
\end{equation*}
Besides, using Lemma \ref{SmallDataScattering}, we obtain that
\begin{equation*}
\Vert \nabla_\g U^\gamma_k\Vert_{L^\frac{10}{3}_{t,x}}^2\lesssim E^1(U^\gamma_k)
\end{equation*}
if $E^1(U^\gamma_k)\le\delta_0$ is sufficiently small. Hence there exists a constant $C=C(E_{max},\eta)$ such that, for all $\gamma$, and all $k$ large enough (depending on $\gamma$),
\begin{equation}\label{SubLinear}
\begin{split}
\Vert \nabla_\g U^\gamma_k\Vert_{L^\frac{10}{3}_{t,x}}^2&\le CE^1(U^\gamma_k)\lesssim_{E_{max},\eta}1\\
\Vert U^\gamma_k\Vert_{L^{10}_{t,x}}^2\lesssim\Vert \nabla_\g U^\gamma_k\Vert_{L^{10}_tL^\frac{30}{13}_x}^2&\le CE^1(U^\gamma_k)\lesssim_{E_{max},\eta}1,
\end{split}
\end{equation}
the second inequality following from H\"older's inequality between the first and the trivial bound $\Vert\nabla_\g U^\gamma_k\Vert_{L^\infty_tL^2_x}\le 2E^1(U^\gamma_k)$.
Now, using \eqref{SubLinear} and Lemma \ref{Assump}, we see that
\begin{equation*}
\begin{split}
\Big\vert\Vert\nabla_\g U^J_{prof,k}\Vert_{L^\frac{10}{3}_{t,x}}^\frac{10}{3}-\sum_{1\le\alpha\le 2J}\Vert\nabla_\g U^\alpha_k&\Vert_{L^\frac{10}{3}_{t,x}}^\frac{10}{3}\Big\vert\le\sum_{1\le\alpha\ne\beta\le 2J}\Vert(\nabla_\g U^\alpha_k)^\frac{7}{3}\nabla_\g U^\beta_k\Vert_{L^1_{t,x}}\\
&\lesssim_{E_{max},\eta} \sum_{1\le \alpha\ne\beta\le 2J}\Vert (\nabla_\g U^\alpha_k)\nabla_\g U^\beta_k\Vert_{L^\frac{5}{3}_{t,x}}\lesssim_{E_{max},\eta} o_k(1).
\end{split}
\end{equation*}
Consequently,
\begin{equation*}
\begin{split}
\Vert\nabla_\g U^J_{prof,k}\Vert_{L^\frac{10}{3}_{t,x}}^\frac{10}{3}
&\le \sum_{1\le\alpha\le 2J}\Vert \nabla_\g U^\alpha_k\Vert_{L^\frac{10}{3}_{t,x}}^\frac{10}{3}+o_k(1)\\
&\lesssim_{E_{max},\eta} C\sum_{1\le\alpha\le 2J}E^1(U^\alpha_k)+o_k(1)\lesssim_{E_{max},\eta}1
\end{split}
\end{equation*}
and using H\"older's inequality and \eqref{BoundedENorm}, we get \eqref{UHasUniformlyBoundedSNorm}.

Using \eqref{BoundedENorm} and \eqref{UHasUniformlyBoundedSNorm} we can apply Proposition \ref{stability} to get $\delta>0$ such that the conclusion of Proposition \ref{stability} holds.

Now, for $F(x)=\vert x\vert^4x$, we have that
\begin{equation*}
\begin{split}
e=\left(i\partial_t+\Delta_\g\right)U^{app}_k-U^{app}_k\vert U^{app}_k\vert^4
&=\sum_{1\le\alpha\le 2J}\left(\left(i\partial_t+\Delta_\g\right)U^\alpha_k-F(U^\alpha_k)\right)\\
&+\sum_{1\le\alpha\le 2J}F(U^\alpha_k)-F(U^{app}_k).
\end{split}
\end{equation*}
The first term is identically $0$, while using Lemma \ref{ControlOfe}, we see that taking $J$ large enough, we can ensure that the second is smaller than $\delta$ given above in $L^2_tH^{1,\frac{6}{5}}_x$-norm for all $k$ large enough. Then, since $u_k(0)=U^{app}_k(0)$, Sobolev's inequality and the conclusion of Proposition \ref{stability} imply that for all $k$ large, and all interval $J$
\begin{equation*}
\Vert u_k\Vert_{Z(J)}\lesssim\Vert u_k\Vert_{S^1(J)}\le \Vert u_k-U^{app}_k\Vert_{S^1(J)}+\Vert U^{app}_k\Vert_{S^1(\mathbb{R})}\lesssim_{E_{max},\eta} 1
\end{equation*}
where we have used \eqref{UHasUniformlyBoundedSNorm}. Then, we see that $u_k$ is global for all $k$ large enough and that $u_k$ has uniformly bounded $Z$-norm, which contradicts \eqref{CondForCompactnessBigSNorm}. This ends the proof.

\subsection{Criterion for linear evolution}

\begin{lemma}\label{SmallDataScattering}
For any $M>0$,
there exists $\delta>0$ such that for any interval $J\subset\mathbb{R}$, if
\begin{equation*}
\begin{split}
\Vert \nabla_\g\phi\Vert_{L^2(\mathbb{H}^3)}&\le M\\
\Vert e^{it\Delta_\g}\phi\Vert_{Z(J)}&\le\delta,
\end{split}
\end{equation*}
then for any $t_0\in J$, the maximal solution $(I,u)$ of \eqref{eq1} satisfying $u(t_0)=e^{it_0\Delta_\g}\phi$ satisfies $J\subset I$ and
\begin{equation}\label{BoundOnuSmallData}
\begin{split}
\Vert u-e^{it\Delta_\g}\phi\Vert_{S^1(J)}&\le \delta^3\\
\Vert u\Vert_{S^1(J)}&\le C(M,\delta).
\end{split}
\end{equation}
Besides, if $J=(-\infty,T)$, then there exists a unique maximal solution $(I,u)$, $J\subset I$ of \eqref{eq1}
such that
\begin{equation}\label{ScatHyp}
\lim_{t\to-\infty}\Vert \nabla_\g\left(u(t)-e^{it\Delta_\g}\phi\right)\Vert_{L^2(\mathbb{H}^3)}=0
\end{equation}
and \eqref{BoundOnuSmallData} holds in this case too. The same statement holds in the Euclidean case when $(\mathbb{H}^3,\g)$ is replaced by $(\mathbb{R}^3,\delta_{ij})$.
\end{lemma}

\begin{proof}[Proof of Lemma \ref{SmallDataScattering}] The first part is a direct consequence of Proposition \ref{stability}. Indeed, let $v=e^{it\Delta_\g}\phi$. Then clearly \eqref{ume} is satisfied while using Strichartz estimates,
\begin{equation*}
\Vert \nabla_\g v\vert v\vert^4\Vert_{L^2_tL^\frac{6}{5}_x(J\times\mathbb{H}^3)}\le \Vert v\Vert_{Z(J)}^4\Vert\nabla_\g e^{it\Delta_\g}\phi\Vert_{L^{10}_tL^\frac{30}{13}_x(J\times\mathbb{H}^3)}\lesssim M\delta^4,
\end{equation*}
thus we get \eqref{safetycheck}. Then we can apply Proposition \ref{stability} with $\rho=1$ to conclude. The second claim is classical and follows from a fixed point argument.

\end{proof}

\subsection{Description of an Euclidean nonlinear profile}

\begin{lemma}\label{GEForEP}

Let $(N_k,t_k,h_k)_k\in\mathcal{F}_e$ and $\phi\in\dot{H}^1(\mathbb{R}^3)$. Let $U_k$ be the solution of \eqref{eq1} such that $U_k(0)=\Pi_{t_k,h_k}(T_{N_k}\phi)$.

\medskip

\noindent (i) For $k$ large enough, $U_k\in C(\mathbb{R}:H^1)$ is globally defined, and
\begin{equation}\label{ControlOnZNormForEP}
\Vert U_k\Vert_{Z(\mathbb{R})}\le 2\tilde{C}(E^1_{\mathbb{R}^3}(\phi)).
\end{equation}

\noindent (ii) There exists an Euclidean solution $u\in C(\mathbb{R}:\dot{H}^1(\mathbb{R}^3))$ of
\begin{equation}\label{EEq}
\left(i\partial_t+\Delta\right)u=u\vert u\vert^4
\end{equation}
with scattering data $\phi^{\pm\infty}$ defined as in \eqref{EScat} such that the following holds, up to a subsequence:
for any $\varepsilon>0$, there exists $T(\phi,\varepsilon)$ such that for all $T\ge T(\phi,\varepsilon)$ there exists $R(\phi,\varepsilon,T)$ such that for all $R\ge R(\phi,\varepsilon,T)$, there holds that
\begin{equation}\label{ProxyEuclHyp}
\Vert U_k-\tilde{u}_k\Vert_{S^1(\vert t-t_k\vert\le TN_k^{-2})}\le\varepsilon,
\end{equation}
for $k$ large enough, where
\begin{equation*}
(\pi_{h_k^{-1}}\tilde{u}_k)(t,x)=N_k^{1/2}\eta(N_k\Psi_{I}^{-1}(x)/R)u(N_k\Psi_{I}^{-1}(x),N_k^2(t-t_k)).
\end{equation*}
In addition, up to a subsequence,
\begin{equation}\label{SmallnessOutsideInteractionRegion}
\Vert U_k\Vert_{L^{10}_tH^{1,\frac{30}{13}}_x\cap L^\frac{10}{3}_tH^{1,\frac{10}{3}}_x(\mathbb{H}^3\times\{N_k^2\vert t-t_k\vert\ge T\})}\le\varepsilon
\end{equation}
and for any $\pm(t-t_k)\ge  TN_k^{-2}$,
\begin{equation}\label{ScatEuclSol}
\Vert \nabla_\g \left(U_k(t)-\Pi_{t_k-t,h_k}T_{N_k}\phi^{\pm\infty}\right)\Vert_{L^2}\le \varepsilon,
\end{equation}
for $k$ large enough (depending on $\phi,\varepsilon,T,R$).
\end{lemma}

\begin{proof}[Proof of Lemma \ref{GEForEP}] In view of Lemma \ref{equiv} (i), we may assume that either $t_k=0$ or that $\lim_{k\to\infty}N_k^2|t_k|=\infty$. We may also assume that $h_k=I$ for any $k$.

If $t_k=0$ for any $k$ then the lemma follows from Lemma \ref{step1} and Corollary \ref{step2}: we let $u$ be the nonlinear Euclidean solution of \eqref{EEq} with $u(0)=\phi$ and notice that for any $\delta>0$ there is $T(\phi,\delta)$ such that
\begin{equation*}
\|\nabla u\|_{L^{10/3}_{x,t}(\mathbb{R}^3\times\{|t|\geq T(\phi,\delta)\})}\leq\delta.
\end{equation*}
The bound \eqref{ProxyEuclHyp} follows for any fixed $T\geq T(\phi,\delta)$ from Lemma \ref{step1}. Assuming $\delta$ is sufficiently small and $T$ is sufficiently large (both depending on $\phi$ and $\varepsilon$), the bounds \eqref{SmallnessOutsideInteractionRegion} and \eqref{ScatEuclSol} then follow from Corollary \ref{step2} (which guarantees smallness of $\mathbf{1}_{\pm}(t)\cdot e^{it\Delta_\g}U_k(\pm N_k^{-2}T(\phi,\delta))$ in $L^{10/3}_tH^{1,10/3}_x(\mathbb{H}^3\times\mathbb{R})$) and Lemma \ref{SmallDataScattering}.

Otherwise, if $\lim_{k\to\infty}N_k^2|t_k|=\infty$, we may assume by symmetry that $N_k^2t_k\to+\infty$. Then we let $u$ be the solution of
\eqref{EEq} such that
\begin{equation*}
\Vert\nabla\left(u(t)-e^{it\Delta}\phi\right)\Vert_{L^2(\mathbb{R}^3)}\to0
\end{equation*}
as $t\to-\infty$ (thus $\phi^{-\infty}=\phi$).
We let $\tilde{\phi}=u(0)$ and  apply the conclusions of the lemma to the frame $(N_k,0,h_k)_k\in\mathcal{F}_e$ and $V_k(s)$, the solution of \eqref{eq1} with initial data $V_k(0)=\pi_{h_k}T_{N_k}\tilde{\phi}$. In particular, we see from the fact that $N_k^2t_k\to+\infty$ and \eqref{ScatEuclSol} that
\begin{equation*}
\Vert V_k(-t_k)-\Pi_{t_k,h_k}T_{N_k}\phi\Vert_{H^1(\mathbb{H}^3)}\to 0
\end{equation*}
as $k\to\infty$. Then, using Proposition \ref{stability}, we see that
\begin{equation*}
\Vert U_k-V_k(\cdot-t_k)\Vert_{S^1(\mathbb{R})}\to 0
\end{equation*}
as $k\to\infty$, and we can conclude by inspecting the behavior of $V_k$. This ends the proof.
\end{proof}

\subsection{Non-interaction of nonlinear profiles}

\begin{lemma}\label{Assump}
Let $\widetilde{\phi}_{\mathcal{O}_k}$ and $\widetilde{\psi}_{\mathcal{O}^\prime_k}$ be two profiles associated to orthogonal frames $\mathcal{O}$ and $\mathcal{O}^\prime$. Let $U_k$ and $U^\prime_k$ be the solutions of the nonlinear equation \eqref{eq1}
such that $U_k(0)=\widetilde{\phi}_{\mathcal{O}_k}$ and $U^\prime_k(0)=\widetilde{\psi}_{\mathcal{O}^\prime_k}$. Suppose also that $E^1(\widetilde{\phi}_{\mathcal{O}_k})<E_{max}-\eta$ (resp. $E^1(\widetilde{\psi}_{\mathcal{O}^\prime_k})<E_{max}-\eta$) if $\mathcal{O}\in\mathcal{F}_h$ (resp. $\mathcal{O}^\prime\in\mathcal{F}_h$). Then, up to a subsequence,

\noindent (i) for any $T\in\mathbb{R}$, there holds that
\begin{equation}\label{OrthoConservedLarget}
\langle\nabla_\g U_k(T),\nabla_\g U^\prime_k(T)\rangle_{L^2\times L^2(\mathbb{H}^3)}\to0
\end{equation}
as $k\to\infty$.

\noindent (ii) \begin{equation}\label{Lem5Claim1}
\Vert U_k\nabla_\g U^\prime_k\Vert_{L^{5}_tL^\frac{15}{8}_x(\mathbb{H}^3\times\mathbb{R})}+\Vert (\nabla_\g U_k)\nabla_\g U^\prime_k\Vert_{L^\frac{5}{3}_{t,x}(\mathbb{H}^3\times\mathbb{R})}\to 0
\end{equation}
as $k\to\infty$.
\end{lemma}

\begin{proof}[Proof of Lemma \ref{Assump}] (i) We fix $\varepsilon_0>0$ sufficiently small. We first consider the case of a hyperbolic frame $\mathcal{O}$. If $\mathcal{O}$ is equivalent to $(1,0,h_k)_k$, we may use Lemma \ref{equiv} to set $t_k=0$ for all $k$. In this case, letting $W$ be the solution of \eqref{eq1} with initial data $W(0)=\phi$, we get by invariance of \eqref{eq1} under the action of $\pi$ that $U_k(T)=\Pi_{0,h_k}(W(T))$.

If $\vert t_k\vert\to\infty$, we may assume that $t_k\to+\infty$ and then, we see from Strichartz estimates that for $k$ sufficiently large
\begin{equation*}
\Vert e^{i(t-t_k)\Delta_\g}\phi\Vert_{Z(-\infty,T+1)}\le\varepsilon_0
\end{equation*}
and applying Lemma \ref{SmallDataScattering}, we get that
\begin{equation*}
\Vert\nabla_\g\left( U_k(T)-\Pi_{t_k-T,h_k}\phi\right)\Vert_{L^2}\le \varepsilon_0^3.
\end{equation*}

\medskip

Now we assume that $\mathcal{O}=(N_k,t_k,h_k)_k$ with $N_k\to\infty$. In the case when $\mathcal{O}$ is equivalent to $(N_k,T,h_k)_k$, i.e. if $N_k^2\vert t_k-T\vert$ remains bounded, up to passing to a subsequence, we may assume that $N_k^2(T-t_k)\to T_0$. Then, applying Lemma \ref{GEForEP}, we see that there exists $R>0$ such that
\begin{equation*}
\Vert U_k(T)-\tilde{u}_k(T)\Vert_{H^1(\mathbb{H}^3)}\lesssim \varepsilon_0/2
\end{equation*}
for $k$ large enough.
In particular, for $k$ large enough,
\begin{equation*}
\Vert U_k(T)-\Pi_{0,h_k}T_{N_k}\big(\eta(\frac{\cdot}{R})u(\cdot,T_0)\big)\Vert_{H^1(\mathbb{H}^3)}\le \varepsilon_0.
\end{equation*}

Finally, if $N_k\vert t-t_k\vert\to\infty$, passing to a subsequence, we may assume that $N_k^2(t-t_k)\to+\infty$. In this case again, we see that there exists $\tilde{\phi}$ such that, for $k$ large enough,
\begin{equation*}
\Vert U_k(T)-\Pi_{t_k-T,h_k}T_{N_k}\tilde{\phi}\Vert_{H^1(\mathbb{H}^3)}\le\varepsilon_0.
\end{equation*} 

Therefore, in all cases, we could (up to an error $\varepsilon_0$) replace $U_k(T)$ and $U^\prime_k(T)$ by a linear profile in new frames $\mathcal{O}_T$, $\mathcal{O}^\prime_T$ with the property that $\mathcal{O}_T$ and $\mathcal{O}^\prime_T$ are orthogonal if and only if $\mathcal{O}$ and $\mathcal{O}^\prime$ are orthogonal. Thus \eqref{OrthoConservedLarget} follows from Lemma \ref{equiv} (ii).

\medskip

(ii) We give a proof that the first norm in \eqref{Lem5Claim1} decays, the claim for the second norm is similar. We fix $\varepsilon>0$. Then, applying Lemma \ref{GEForEP} if $U_k$ is a profile associated to a Euclidean frame (respectively \eqref{GDForHypP} if $U_k$ is a profile associated to a Euclidean frame), we see that
\begin{equation*}
\Vert U_k\Vert_{S^1}+\Vert U^\prime_k\Vert_{S^1}\le M<+\infty
\end{equation*}
and that 
there exist $R$ and $\delta$ such that
\begin{equation}\label{alex9}
\begin{split}
&\|\nabla_\g U_k\|_{L^{10}_tL^{30/13}_x\cap L^{10/3}_{x,t}((\mathbb{H}^3\times\mathbb{R})\setminus \mathcal{S}^R_{N_k,t_k,h_k})}+\|U_k\|_{L^{10}_{x,t}((\mathbb{H}^3\times\mathbb{R})\setminus \mathcal{S}^R_{N_k,t_k,h_k})}\le\varepsilon,\\
&\sup_{S,h}\big[\Vert\nabla_\g U_k\Vert_{L^{10}_tL^{30/13}_x\cap L^{10/3}_{x,t}(S^\delta_{N_k,S,h})}+\Vert U_k\Vert_{L^{10}_{x,t}(S^\delta_{N_k,S,h})}\big]\le\varepsilon,
\end{split}
\end{equation}
where
\begin{equation}
\label{DefOfSetS}
\mathcal{S}^a_{N,T,h}:=\{(x,t)\in\mathbb{H}^3\times\mathbb{R}:d_\g(h^{-1}\cdot x,{\bf{0}})\le aN^{-1}\text{ and }\vert t-T\vert\le a^2N^{-2}\}.
\end{equation}
A similar claim holds for $U^\prime_k$ with the same values of $R$, $\delta$.

If $N_k/N^\prime_k\to\infty$, then for $k$ large enough we estimate
\begin{equation*}
\begin{split}
\Vert U_k\nabla_\g U^\prime_k&\Vert_{L^{5}_tL^\frac{30}{16}_x}
\le \Vert U_k\nabla_\g U^\prime_k\Vert_{L^5_tL^\frac{30}{16}_x(\mathcal{S}_{N_k,t_k,h_k}^R)}+
\Vert U_k\nabla_\g U^\prime_k\Vert_{L^5_tL^\frac{30}{16}_x((\mathbb{H}^3\times\mathbb{R})\setminus \mathcal{S}_{N_k,t_k,h_k}^R)}\\
&\le \Vert U_k\Vert_{L^{10}_{t,x}}\Vert\nabla_\g U^\prime_k\Vert_{L^{10}_tL^\frac{30}{13}(\mathcal{S}^\delta_{N'_k,t_k,h_k})}+
\Vert U_k\Vert_{L^{10}_{t,x}((\mathbb{H}^3\times\mathbb{R})\setminus \mathcal{S}_{N_k,t_k,h_k}^R)}\Vert\nabla_\g U'_k\Vert_{L^{10}_tL^\frac{30}{13}_x}\\
&\lesssim_{M}\varepsilon.
\end{split}
\end{equation*}
The case when $N^\prime/N_k\to\infty$ is similar.

Otherwise, we can assume that $C^{-1}\le N_k/N^\prime_k\le C$ for all $k$, and then find $k$ sufficiently large so that $\mathcal{S}^R_{N_k,t_k,h_k}\cap \mathcal{S}^R_{N'_k,t'_k,h'_k}=\emptyset$. Using \eqref{alex9} it follows as before that
\begin{equation*}
\Vert U_k\nabla_\g U^\prime_k\Vert_{L^{5}_tL^\frac{30}{16}_x}\lesssim_M \varepsilon.
\end{equation*}

Hence, in all cases,
\begin{equation*}
\limsup_{k\to\infty}\Vert U_k\nabla_\g U^\prime_k\Vert_{L^5_tL^\frac{15}{8}_x}\lesssim_{M}\varepsilon.
\end{equation*}
The convergence to $0$ of the first term in \eqref{Lem5Claim1} follows.
\end{proof}

\subsection{Control of the error term}

\begin{lemma}\label{ControlOfe}

With the notations in the proof of Proposition \ref{lem4}, there holds that

\begin{equation}\label{2Terms}
\begin{split}
\lim_{J\to\infty}\limsup_{k\to\infty}\Big\| \nabla_\g\big(F(U^{app}_k)-\sum_{1\le\alpha\le 2J}F(U^\alpha_k)\big)\Big\|_{L^2_tL^\frac{6}{5}_x}=0.
\end{split}
\end{equation}

\end{lemma}

\begin{proof}

Fix $\varepsilon_0>0$.
For fixed $J$, we let
\begin{equation*}
U^J_{prof,k}=\sum_{1\le\mu\le J}U^\mu_{e,k}+\sum_{1\le\nu\le J}U^\nu_{h,k}
=\sum_{1\le\gamma\le 2J}U^\gamma_k
\end{equation*}
be the sum of the profiles. Then we separate
\begin{equation*}
\begin{split}
&\Big\| \nabla_\g\big(F(U^{app}_k)-\sum_{1\le\alpha\le 2J}F(U^\alpha_k)\big)\Big\|_{L^2_tL^\frac{6}{5}_x}\\
&\leq\Big\|\nabla_\g \big(F(U^{app}_k)-F(U^J_{prof,k})\big)\Big\|_{L^2_tL^\frac{6}{5}_x}+
\Big\|\nabla_\g\big( F(U^J_{prof,k})-\sum_{1\le\alpha\le 2J}F(U^\alpha_k)\big)\Big\|_{L^2_tL^\frac{6}{5}_x}.
\end{split}
\end{equation*}

We first claim that, for fixed $J$,
\begin{equation}\label{2Terms1}
\limsup_{k\to\infty}
\Big\|\nabla_\g( F(U^J_{prof,k})-\sum_{1\le\alpha\le 2J}F(U^\alpha_k))\Big\|_{L^2_tL^\frac{6}{5}_x}=0.
\end{equation}
Indeed, using that
\begin{equation*}
\vert\nabla_\g(F(\sum_{1\le\alpha\le 2J}U^\alpha_k)-\sum_{1\le\alpha\le 2J}F(U^\alpha_k))\vert\lesssim\sum_{\alpha\ne\beta,\gamma}\vert U^\gamma_k\vert^3\vert U^\alpha_k\nabla_\g U^\beta_k\vert,
\end{equation*}
we see that
\begin{equation*}\label{InterClaimPP2}
\Big\|\nabla_\g( F(U^J_{prof,k})-\sum_{1\le\alpha\le 2J}F(U^\alpha_k))\Big\|_{L^2_tL^\frac{6}{5}_x}
\lesssim\sum_{\alpha\ne\beta,\gamma}\Vert U^\gamma_k\Vert_{L^{10}_{t,x}}^3\Vert U^\alpha_k\nabla_\g U^\beta_k\Vert_{L^5_tL^\frac{15}{8}_x}.
\end{equation*}
Therefore \eqref{2Terms1} follows from \eqref{Lem5Claim1} since the sum is over a finite set and each profile is bounded in $L^{10}_{t,x}$ by \eqref{SubLinear}.

Now we prove that, for any given $\varepsilon_0>0$,
\begin{equation}\label{2Terms2}
\limsup_{J\to\infty}\limsup_{k\to\infty}\Big\|\nabla_\g \big(F(U^{app}_k)-F(U^J_{prof,k})\big)\Big\|_{L^2_tL^\frac{6}{5}_x}\lesssim\varepsilon_0.
\end{equation}
This would complete the proof of \eqref{2Terms}. We first remark that, from \eqref{UHasUniformlyBoundedSNorm}, $U^J_{prof,k}$ has bounded $L^{10}_tH^{1,\frac{30}{13}}_x$-norm, uniformly in $J$ for $k$ sufficiently large. We also let $j_0=j_0(\varepsilon_0)$ independent of $J$ be such that\footnote{The fact that $j_0$ exists follows from \eqref{AlmostOrtho} and \eqref{SubLinear}.}
\begin{equation}\label{ChoiceOfj0}
\sup_{\alpha\ge j_0}\limsup_{k\to\infty}\Vert U^\alpha_k\Vert_{L^{10}_{t,x}}\lesssim\varepsilon_0.
\end{equation}
\medskip

Now we compute
\begin{equation*}
\begin{split}
\Vert\nabla_\g&\left( F(U^J_{prof,k}+e^{it\Delta_\g}r^J_k)-F(U^J_{prof,k})\right)\Vert_{L^2_tL^\frac{6}{5}_x}\\
&\lesssim \sum_{j=1}^5\sum_{p=0}^1\Vert \nabla_\g^p (e^{it\Delta_\g}r^J_k)^j\nabla_\g^{1-p}(U^J_{prof,k})^{5-j}\Vert_{L^2_tL^\frac{6}{5}_x}.
\end{split}
\end{equation*}
Since both $U^J_{prof,k}$ and $e^{it\Delta_\g}r^J_k$ are bounded in $L^{10}_tH^{1,\frac{30}{13}}_x$ uniformly in $J$, if there is at least one term $e^{it\Delta_\g}r^J_k$ with no derivative, we can bound the norm in the expression above by
\begin{equation*}
\Vert \nabla_\g^p (e^{it\Delta_\g}r^J_k)^j\nabla_\g^{1-p}(U^J_{prof,k})^{5-j}
\Vert_{L^2_tL^\frac{6}{5}_x}\lesssim_{E_{max},\eta} \Vert e^{it\Delta_\g}r^J_k\Vert_{L^{10}_{t,x}}
\end{equation*}
uniformly in $J$, so that taking the limit $k\to\infty$ and then $J\to\infty$, we get $0$.
Hence we need only consider the term
\begin{equation*}
\Vert (U^J_{prof,k})^4\nabla_\g (e^{it\Delta_\g}r^J_k)\Vert_{L^{2}_tL^\frac{6}{5}_x}.
\end{equation*}

Expanding further $(U^J_{prof,k})^4$, and using Lemma \ref{Assump} (ii) and \eqref{SubLinear}, we see that
\begin{equation*}
\begin{split}
\limsup_{k\to\infty}\Vert (U^J_{prof,k})^4\nabla_\g (e^{it\Delta_\g}r^J_k)\Vert_{L^{2}_tL^\frac{6}{5}_x}
&=\limsup_{k\to\infty}\sum_{1\le\alpha\le J}\Vert (U^\alpha_k)^4\nabla_\g (e^{it\Delta_\g}r^J_k)\Vert_{L^{2}_tL^\frac{6}{5}_x}\\
&\lesssim \limsup_{k\to\infty}\sum_{1\le\alpha\le J}\Vert U^\alpha_k\Vert_{L^{10}_{t,x}}^3\Vert U^\alpha_k\nabla_\g (e^{it\Delta_\g}r^J_k)\Vert_{L^5_tL^\frac{15}{8}_x}\\
&\lesssim_{E_{max},\eta} \limsup_{k\to\infty}\sum_{1\le\alpha\le j_0}E^1( U^\alpha_k)\Vert U^\alpha_k\nabla_\g (e^{it\Delta_\g}r^J_k)\Vert_{L^5_tL^\frac{15}{8}_x}\\
&+\limsup_{k\to\infty}\sum_{j_0\le\alpha\le J}E^1( U^\alpha_k)\Vert U^\alpha_k\Vert_{L^{10}_{t,x}}\Vert\nabla_\g (e^{it\Delta_\g}r^J_k)\Vert_{L^{10}_tL^\frac{30}{13}_x}
\\
\end{split}
\end{equation*}
where $j_0$ is chosen in \eqref{ChoiceOfj0}. Consequently, using the summation formula for the energies \eqref{AlmostOrtho}, we get
\begin{equation*}
\limsup_{k\to\infty}\Vert (U^J_{prof,k})^4\nabla_\g (e^{it\Delta_\g}r^J_k)\Vert_{L^{2}_tL^\frac{6}{5}_t}\lesssim_{E_{max},\eta}\varepsilon_0+\sup_{1\le\alpha\le j_0}\limsup_{k\to\infty}\Vert U^\alpha_k\nabla_\g (e^{it\Delta_\g}r^J_k)\Vert_{L^5_tL^\frac{15}{8}_x}.
\end{equation*}

Finally, we obtain from Lemma \ref{locsmo} that for any profile $U^\alpha_k$,
\begin{equation}\label{CsqOfLocSmo}
\lim_{J\to\infty}\limsup_{k\to\infty}\Vert U^\alpha_k\nabla_\g (e^{it\Delta_\g}r^J_k)\Vert_{L^5_tL^\frac{15}{8}_x(\mathbb{H}^3\times\mathbb{R})}=0.
\end{equation}
This would imply \eqref{2Terms2} and hence the proof of Lemma \ref{ControlOfe}. To prove \eqref{CsqOfLocSmo}, fix $\varepsilon>0$. For $U^\alpha_k$ given, we consider the sets $\mathcal{S}_{N,T,h}^a$ as defined in \eqref{DefOfSetS}. For $R$ large enough we have, using \eqref{alex9},
\begin{equation*}
\begin{split}
&\Vert U^\alpha_k\nabla_\g (e^{it\Delta_\g}r^J_k)\Vert_{L^5_tL^\frac{15}{8}_x((\mathbb{H}^3\times\mathbb{R})\setminus\mathcal{S}^R_{N_k,t_k,h_k})}\\
&\le \Vert U^\alpha_k\Vert_{L^{10}_{x,t}((\mathbb{H}^3\times\mathbb{R})\setminus\mathcal{S}^R_{N_k,t_k,h_k})}\Vert \nabla_\g (e^{it\Delta_\g}r^J_k)\Vert_{L^{10}_tL^\frac{30}{13}_x}\lesssim_{E_{max},\eta}\varepsilon.
\end{split}
\end{equation*}

Now in the case of a hyperbolic profile $U^\nu_{h,k}$, we know that $W^\nu$ as in \eqref{DescHypProfile1} satisfies $W^\nu\in L^{10}_{x,t}(\mathbb{H}^3\times\mathbb{R})$. We choose $W^{\nu,\prime}\in C^\infty_c(\mathbb{H}^3\times\mathbb{R})$ such that
\begin{equation*}
\Vert W^\nu-W^{\nu,\prime}\Vert_{L^{10}_{x,t}(\mathbb{H}^3\times\mathbb{R})}\le\varepsilon.
\end{equation*}
Using \eqref{GDForHypP} we see that there exists a constant $C_{\nu,\varepsilon}$ such that
\begin{equation*}
\begin{split}
\Vert U^\nu_{h,k}\nabla_\g (e^{it\Delta_\g}r^J_k)\Vert_{L^5_tL^\frac{15}{8}_x(\mathcal{S}^R_{N_k,t_k,h_k})}
&\le \Vert (U^\nu_{h,k}-\pi_{h_k^\nu}W^{\nu,\prime}(.-t^\nu_k))\nabla_\g (e^{it\Delta_\g}r^J_k)\Vert_{L^5_tL^\frac{15}{8}_x(\mathcal{S}^R_{N_k,t_k,h_k})}\\
&+\Vert W^{\nu,\prime}\Vert_{L^\infty_{t,x}}\Vert \nabla_\g (e^{it\Delta_\g}r^J_k)\Vert_{L^5_tL^\frac{15}{8}_x(\mathcal{S}^R_{N_k,t_k,h_k})}\\
&\lesssim_{E_{max},\eta} \varepsilon+C_{\nu,\varepsilon}\Vert \nabla_\g (e^{it\Delta_\g}r^J_k)\Vert_{L^5_tL^\frac{15}{8}_x(\mathcal{S}^R_{N_k,t_k,h_k})}.
\end{split}
\end{equation*}

In the case of a Euclidean profile, we choose $v\in C^\infty_c(\mathbb{R}^3\times\mathbb{R})$ such that
\begin{equation*}
\Vert u-v\Vert_{L^{10}_{t,x}(\mathbb{R}^3\times\mathbb{R})}\le\varepsilon,
\end{equation*}
for $u$ given in Lemma \ref{GEForEP}. Then, using \eqref{ProxyEuclHyp}, we estimate as before
\begin{equation*}
\Vert U^\mu_{e,k}\nabla_\g (e^{it\Delta_\g}r^J_k)\Vert_{L^5_tL^\frac{15}{8}_x(\mathcal{S}^R_{N_k,t_k,h_k})}\lesssim_{E_{max},\eta}\varepsilon+C_{\mu,\varepsilon}(N_k^\mu)^\frac{1}{2}\Vert\nabla_\g (e^{it\Delta_\g}r^J_k)\Vert_{L^5_tL^\frac{15}{8}_x(\mathcal{S}^R_{N_k,t_k,h_k})}.
\end{equation*}

Therefore, we conclude that in all cases,
\begin{equation*}
\Vert U^\alpha_k\nabla_\g (e^{it\Delta_\g}r^J_k)\Vert_{L^5_tL^\frac{15}{8}_x(\mathcal{S}^R_{N_k,t_k,h_k})}\lesssim_{E_{max},\eta}\varepsilon+C_{\alpha,\varepsilon}
(N_k^\alpha)^\frac{1}{2}\Vert\nabla_\g (e^{it\Delta_\g}r^J_k)\Vert_{L^5_tL^\frac{15}{8}_x(\mathcal{S}^R_{N_k,t_k,h_k})}.
\end{equation*}
Finally we use Lemma \ref{locsmo} and \eqref{rkJSmall} to conclude that
\begin{equation*}
\lim_{J\to\infty}\limsup_{k\to\infty}\Vert U^\alpha_k\nabla_\g (e^{it\Delta_\g}r^J_k)\Vert_{L^5_tL^\frac{15}{8}_x(\mathcal{S}^R_{N_k,t_k,h_k})}\lesssim_{E_{max},\eta}\varepsilon.
\end{equation*}
Since $\varepsilon$ was arbitrary, we obtain \eqref{CsqOfLocSmo} and hence finish the proof.
\end{proof}


\begin{thebibliography}{99}

\bibitem{An} J.-P. Anker, $L_p$ Fourier multipliers on Riemannian symmetric spaces of the noncompact type, Ann. of Math. (2) {\bf{132}} (1990), 597--628.

\bibitem{AnPi} J.-P. Anker and V. Pierfelice, Nonlinear Schr\"{o}dinger equation on real hyperbolic spaces, Ann. Inst. H. Poincar\'{e} Anal. Non Lin\'{e}aire {\bf{26}}  (2009), 1853--1869.

\bibitem{BaGe} H. Bahouri and P. G\'erard, Patrick, High frequency approximation of solutions to critical nonlinear wave equations.  Amer. J. Math.  {\bf 121}  (1999),  no. 1, 131--175. 

\bibitem{BaSh} H. Bahouri and J. Shatah, Decay estimates for the critical semilinear wave equation, Ann. Inst. H. Poincar\'{e} Anal. Non Lin\'{e}aire {\bf{15}} (1998), 783--789.

\bibitem{Ba} V. Banica, The nonlinear Schr\"{o}dinger equation on the hyperbolic space,  Comm. Partial Differential Equations {\bf 32} (2007), 1643--1677.

\bibitem{BaCaDu} V. Banica, R. Carles, T. Duyckaerts, On scattering for NLS: from Euclidean to hyperbolic space, Discrete Contin. Dyn. Syst. {\bf{24}} (2009), 1113--1127.

\bibitem{BaCaSt} V. Banica, R. Carles and G. Staffilani, Scattering theory for radial nonlinear Schr\"{o}dinger equations on hyperbolic space, Geom. Funct. Anal. {\bf{18}} (2008), 367--399.

\bibitem{BaDu} V. Banica and T. Duyckaerts, Weighted Strichartz estimates for radial Schr\"{o}dinger equation on noncompact manifolds, Dyn. Partial Differ. Equ. {\bf{4}} (2007), 335--359.

\bibitem{Bouc} J.M. Bouclet, Strichartz estimates on asymptotically hyperbolic manifolds, Analysis and P.D.E., to appear.

\bibitem{Bo2} J. Bourgain, Fourier transform restriction phenomena for certain lattice subsets and applications to nonlinear evolution equations. I. Schr\"{o}dinger equations, Geom. Funct. Anal. {\bf{3}} (1993), 107--156.

\bibitem{Bo1} J. Bourgain, Exponential sums and nonlinear Schr\"{o}dinger equations, Geom. Funct. Anal. {\bf{3}}  (1993), 157--178.

\bibitem{B} J. Bourgain, Global wellposedness of defocusing critical nonlinear Schr\"{o}dinger equation in the radial case, J. Amer. Math. Soc. {\bf {12}} (1999), 145--171.

\bibitem{Br} W. O. Bray, Aspects of harmonic analysis on real hyperbolic space.  Fourier analysis (Orono, ME, 1992), 77--102, Lecture Notes in Pure and Appl. Math. 157, Dekker, New York, 1994.

\bibitem{BuGeTz} N. Burq, P. G\'{e}rard, and N. Tzvetkov, Strichartz inequalities and the nonlinear Schr\"{o}dinger equation on compact manifolds, Amer. J. Math. {\bf{126}} (2004), 569--605.

\bibitem{BuGeTz2} N. Burq, P. G\'{e}rard, and N. Tzvetkov, Bilinear eigenfunction estimates and the nonlinear Schr\"{o}dinger equation on surfaces, Invent. Math. {\bf{159}} (2005), 187--223.

\bibitem{BuLePl} N. Burq, G. Lebeau, and F. Planchon, Global existence for energy critical waves in 3-D domains, J. Amer. Math. Soc. {\bf{21}} (2008), 831--845.

\bibitem{BuPl} N. Burq and F. Planchon, Global existence for energy critical waves in 3-D domains: Neumann boundary conditions, Amer. J. Math.  {\bf{131}}  (2009), 1715--1742.

\bibitem{Cazenave:book} T. Cazenave, Semilinear Schr\"odinger Equations, Courant Lecture Notes in Mathematics 10, New York University, Courant Institute of Mathematical Sciences, New  York; American Mathematical Society, Providence, RI, 2003.

\bibitem{ChrMar} H. Christiansen, J. Marzuola, Existence and Stability of Solitons for the Nonlinear Schrodinger Equation on Hyperbolic Space, Nonlinearity {\bf 23} (2010), no. 1, 89--106.

\bibitem{ClSt} J.-L. Clerc and E. M. Stein, $L^p$-multipliers for noncompact symmetric spaces, Proc. Nat. Acad. Sci. U.S.A {\bf{71}} (1974), 3911--3912.

\bibitem{CKSTTcrit} J. Colliander,  M. Keel, G.  Staffilani, H. Takaoka and  T. Tao, Global well-posedness and scattering for the energy-critical nonlinear Schr\"{o}dinger equation in $\R^3$, Ann. of Math {\bf{167}} (2008), 767--865.

\bibitem{CKSTTTorus} J. Colliander,  M. Keel, G.  Staffilani, H. Takaoka and  T. Tao, Transfer of energy to high frequencies in the cubic defocusing nonlinear Schr\"odinger equation, Invent. Math., to appear.

\bibitem{Do} S. -I. Doi, Smoothing effects of Schr\"{o}dinger evolution groups on Riemannian manifolds, Duke Math. J. {\bf{82}} (1996), 679--706.

\bibitem{DruHebRob} O. Druet, E. Hebey, and F. Robert, Blow-up theory for elliptic PDEs in Riemannian geometry, Mathematical Notes {\bf{45}}, Princeton University Press, 2004.

\bibitem{GerPie} P. G\'erard and V. Pierfelice, Nonlinear Schr\"odinger equation on four-dimensional compact manifolds.  Bull. Soc. Math. France  {\bf 138}  (2010),  no. 1, 119--151.

\bibitem{GidSpr} B. Gidas and J. Spruck, A priori bounds for positive solutions of nonlinear elliptic equations, Comm. Partial Differential Equations {\bf{6}} (1981), 883-901

\bibitem{Gr1} M. Grillakis, Regularity and asymptotic behaviour of the wave equation with a critical nonlinearity, Ann. of Math. {\bf{132}} (1990), 485--509.

\bibitem{Gr2} M. Grillakis, Regularity for the wave equation with a critical nonlinearity. Comm. Pure Appl. Math. {\bf{45}} (1992), 749--774.

\bibitem{G} M. Grillakis,  On nonlinear Schr\"{o}dinger equations, Comm. Partial Differential Equations {\bf {25}} (2000), 1827--1844.

\bibitem{HebVau} E. Hebey and M. Vaugon, The best constant problem in the Sobolev embedding theorem for complete Riemannian manifolds, Duke Math. J. {\bf{79}} (1995), 235--279. 

\bibitem{He3} S. Helgason, Radon-Fourier transform on symmetric spaces and related group representations, Bull. Amer. Math. Soc. {\bf{71}} (1965), 757--763.

\bibitem{He2} S. Helgason, Geometric analysis on symmetric spaces, Mathematical Surveys and Monographs 39, American Mathematical Society, Providence, RI, 1994.

\bibitem{HeTaTz} S. Herr, D. Tataru, and N. Tzvetkov, Global well-posedness of the energy critical nonlinear Schrödinger equation with small initial data in $H^1(T^3)$, Preprint (2010).

\bibitem{IbMa} S. Ibrahim and M. Majdoub, Solutions globales de l'\'{e}quation des ondes semi-lin\'{e}aire critique \'{a} coefficients variables. (French) [Global solutions of the critical semilinear wave equation with variable coefficients]  Bull. Soc. Math. France {\bf{131}} (2003), 1--22.

\bibitem{IbMaMaNa} S. Ibrahim, M. Majdoub, N. Masmoudi, and K. Nakanishi, Scattering for the two-dimensional energy-critical wave equation, Duke Math. J. {\bf{150}} (2009), 287--329.

\bibitem{IbMaMaNa2} S. Ibrahim, N. Masmoudi, and K. Nakanishi, Scattering threshold for the focusing nonlinear Klein-Gordon equation, Preprint (2010).

\bibitem{IoSt} A. Ionescu and G. Staffilani, Semilinear Schr\"{o}dinger flows on hyperbolic spaces: scattering in $H^1$, Math. Ann {\bf{345}} (2009), 133--158.

\bibitem{Ka} L. Kapitanskii, Global and unique weak solutions of nonlinear wave equations, Math. Res. Lett. {\bf{1}} (1994), 211–-223.

\bibitem{KeTa} M. Keel and T. Tao, Endpoint Strichartz estimates, Amer. J. Math. {\bf{120}} (1998), 955--980.

\bibitem{KeMe} C. E. Kenig and F. Merle, Global well-posedness, scattering and blow-up for the energy-critical, focusing, non-linear Schr\"{o}dinger equation in the radial case, Invent. Math. {\bf{166}} (2006), 645--675.

\bibitem{KeMe2} C. E. Kenig and F. Merle, Global well-posedness, scattering and blow-up for the energy-critical focusing non-linear wave equation, Acta Math. {\bf{201}} (2008), 147--212.

\bibitem{Ker} S. Keraani, On the defect of compactness for the Strichartz estimates of the Schr\"{o}dinger equations.  {\it J. Differential Equations}  {\bf{175}}  (2001), 353--392.

\bibitem{KilStoVis} R. Killip, B. Stovall and M. Visan, Scattering for the cubic Klein--Gordon equation in two space dimensions, Preprint (2010).

\bibitem{La} C. Laurent, On stabilization and control for the critical Klein--Gordon equation on a 3-D compact manifold, Preprint (2010).

\bibitem{Pie} V. Pierfelice, Weighted Strichartz estimates for the Schr\"odinger and wave equations on Damek-Ricci spaces.  Math. Z.  260  (2008),  no. 2, 377--392. 

\bibitem{RV} E. Ryckman and M. Visan,  Global well-posedness and scattering for the defocusing energy-critical nonlinear Schr\"{o}dinger equation in $\R\sp {1+4}$, Amer. J. Math. {\bf {129}} (2007), 1--60.

\bibitem{Sch} R. Schoen, Variational theory for the total scalar curvature functional for Riemannian metrics and related topics, Topics in Calculus of Variations, Lecture Notes in Math., 1989.

\bibitem{ShSt1} J. Shatah and M. Struwe, Regularity results for nonlinear wave equations, Ann. of Math. {\bf{138}} (1993), 503–-518. 

\bibitem{ShSt2} J. Shatah and M. Struwe, Well-posedness in the energy space for semilinear wave equations with critical growth, Int. Math. Res. Notices {\bf{1994}} (1994), 303--309.

\bibitem{StTo} R. J. Stanton and  P. A. Tomas, Expansions for spherical functions on noncompact symmetric spaces, Acta Math. {\bf{140}} (1978), 251--271.

\bibitem{Strauss} W. Strauss, Nonlinear Scattering theory,  Scattering theory in mathematical physics (J. Lavita and J. P. Marchand, eds), Reidel, 1974.

\bibitem{St} M. Struwe, Globally regular solutions to the $u^5$ Klein--Gordon equation, Ann. Scuola Norm. Sup. Pisa Cl. Sci. {\bf{15}} (1989), 495--513.

\bibitem{Tcrit} T. Tao, Global well-posedness and scattering for the higher-dimensional energy-critical nonlinear Schr\"odinger equation for radial data, New York J. Math. {\bf{11}} (2005), 57--80.

\bibitem{Tao:book} T. Tao, Nonlinear Dispersive Equations. Local and Global Analysis, CBMS Regional Conference Series in Mathematics, {\bf{106}}, American Mathematical Society, Providence, RI, 2006.

\bibitem{V} M.  Visan, The defocusing energy-critical nonlinear Schr\"{o}dinger equation in higher dimensions, Duke Math. J. {\bf {138}} (2007), no. 2, 281--374.
 
\end{thebibliography}
\end{document}